\documentclass[12pt]{article}

\usepackage{amsmath,
amsfonts,latexsym,graphicx,amssymb,amsthm}
\usepackage{csquotes}
\usepackage{cjhebrew}

\usepackage{color}
\usepackage{authblk}
\usepackage[new]{old-arrows}

\newcommand{\na}{\mathbb{A}}
\newcommand{\C}{\mathbb{C}}
\newcommand{\F}{\mathbb{F}}

\renewcommand{\P}{\mathbb{P}}
\newcommand{\Q}{\mathbb{Q}}

\newcommand{\Z}{\mathbb{Z}}

\newcommand{\Ad}{\mathrm{Ad}}
\newcommand{\Aff}{\mathrm{Aff}}
\newcommand{\Aut}{\mathrm{Aut}}
\newcommand{\SAut}{\mathrm{SAut}}
\newcommand{\TAut}{\mathrm{TAut}}
\newcommand{\Card}{\mathrm{Card\,}}

\newcommand{\Elem}{\mathrm{Elem}}
\newcommand{\SElem}{\mathrm{SElem}}

\newcommand{\End}{\mathrm{End}}

\newcommand{\GL}{{\mathrm {GL}}}

\newcommand{\hc}{{\mathrm{hc}}}

\newcommand{\id}{\mathrm{id}}

\newcommand{\Image}{\mathrm{Im}}

\newcommand{\jac}{\mathrm{Jac}}
\newcommand{\Ker}{\mathrm{Ker}}

\newcommand{\SL}{{\mathrm {SL}}}
\newcommand{\St}{{\mathrm {St}}}

\newcommand{\rk}{\mathrm{rk}}

\newcommand{\Spec}{{\mathrm {Spec}}\,}

\renewcommand{\d}{\mathrm{d}}
\renewcommand{\mod}{\,\mathrm{mod}\,}

\newtheorem*{MainA}{Theorem A}
\newtheorem*{MainB}{Theorem B}
\newtheorem*{MainC}{Theorem C}
\newtheorem*{MainD}{Theorem D}

\newtheorem*{MainA1}{Theorem A.1}
\newtheorem*{MainA2}{Theorem A.2}

\newtheorem*{Cor1.1}{Corollary 1.1}
\newtheorem*{Cor1.2}{Corollary 1.2}

\newtheorem*{CorD1}{Corollary D.1}
\newtheorem*{CorD2}{Corollary D.2}

\newtheorem*{vdKthm}{van der Kulk Theorem}
\newtheorem*{Cornu}{Cornulier Theorem}

\newtheorem{lemma}{Lemma}%[section]
\newcommand{\ch}{\operatorname{ch}}

\font\small=cmr10

\MakeOuterQuote{"}

\begin{document}

\title{Linearity and Nonlinearity of Groups of Polynomial Automorphisms of the Plane}

\author{Olivier Mathieu}

\affil{\small Institut Camille Jordan du CNRS\\

\small Universit\'e de Lyon\\ 

\small F-69622 Villeurbanne Cedex\\

\small mathieu@math.univ-lyon1.fr}

%\begin{document}

\maketitle

\begin{abstract} Let $K$ be a field, and let $\Aut \,K^2$ be the group of polynomial automorphisms of $K^2$. If $K$ is infinite, this group is nonlinear.
Moreover it contains nonlinear FG subgroups when 
$\ch\,K=0$. On the opposite, it contains some
linear "finite codimension" subgroups. This
phenomenon  is specific to dimension two: it is also proved that "finite codimension" subgroups  of
$\Aut\,K^3$ are nonlinear, even for a finite field $K$.
\footnote{Research supported by UMR 5208 du CNRS}

\end{abstract}

\noindent
\centerline{\it This paper is respectfully dedicated to 
Jacques Tits.}

\section*{Introduction}

Recall that a group $\Gamma$ is called {\it linear}, 
or {\it linear over a ring} in case of ambiguity, if
there is an embedding $\Gamma\subset \GL(n,R)$
for some integer $n$ and some commutative ring $R$. 
Moreover $\Gamma$ is called {\it linear over a field} 
if it can be embedded into $\GL(n,K)$ for some integer $n$ and some field $K$.

Let $\Aut\,K^2$ be the group of polynomial automorphisms  
of the affine plane $K^2$. In this paper, we will investigate the linearity or nonlinearity properties of the subgroups 
of $\Aut\,K^2$. In particular, we will  consider the following subgroups

\centerline{$\Aut_0\,K^2=\{\phi\in \Aut\,K^2\vert\,\phi({\bf 0})={\bf 0})\}$, }

\centerline{$\SAut\,K^2=\{\phi\in \Aut\,K^2\vert\,\jac(\phi)=1\}$,}

\centerline{$\SAut_0\,K^2=\SAut\,K^2\cap \Aut_0\,K^2$, and}

\centerline{$\Aut_1\,K^2=\{\phi\in \Aut_0\,K^2\vert\,{\textnormal d}
\phi\vert_{\bf 0}=\id\}$,}

\noindent where $\jac(\phi):=\det\,\d\phi$ is the jacobian of $\phi$. The first result is

\begin{MainA} (A.1) If $K$ is infinite, the group $\SAut_0\,K^2$ is not linear, even over a ring.

(A.2) Moreover if $\ch\,K=0$
the group $\Aut_0\,K^2$ contains finitely generated subgroups which are  not linear,
even over a ring.
\end{MainA}

It was known that the much larger Cremona group $Cr_2(\Q)$ is
not linear over a field, see \cite{C} \cite{P}. 
More recently, the nonlinearity over a field of the whole group
$\Aut\,\Q^2$ was proved in \cite {Co}. In the same paper, Y. Cornulier raised the question (answered by Theorem A) 
of finding a nonlinear finitely generated (FG in the sequel) 
subgroup in $\Aut\,\Q^2$. 

Various authors show that the automorphism groups
of algebraic varieties share some properties with linear groups, see e.g. 
\cite{S00}\cite {BPZ}. More specifically, it was proved in  \cite{L} that
$\Aut\,K^2$ satisfies Tits alternative. Theorem A shows that these results are not a consequence of  classical 
results for linear groups. 

On the  opposite there is

\begin{MainB}
For any  field $K$, the group $\Aut_1\,K^2$ is
 linear over $K(t)$. 
 
 Moreover if
 $K\supset k(t)$ for some infinite field $k$, then
 there exists an embedding 
 $\Aut_1\,K^2\subset \SL(2,K)$.
\end{MainB}

\noindent 
Note that Theorem B implies that 
 $\Aut\,K^2$ is linear when $K$ is finite.
The previous results suggest to ask

\centerline{\it which groups $G$ containing $\Aut_1\,K^2$ are linear?}
 
\noindent  If $G\not\subset \Aut_0\,K^2$, is easy to show that  $G$ contains $\SAut\,K^2$, an therefore
$G$ is nonlinear except if $K$ is finite. Otherwise
$G$ is isomorphic to 
$\Aut_S\,K^2:=\{\phi\in\Aut_0\,K^2\,\vert\, \d\phi_{\bf 0}\in S\}$, where $S$ is a subgroup of $\GL(2,K)$.
Therefore the previous question can be reformulate as

\centerline{\it for which subgroups $S$  of $\GL(2,K)$, is the group
$\Aut_S\,K^2$  linear?}

\noindent Theorem C provides a complicated condition 
${\cal G}$ for the linearity. For example, it shows that 
the group
$\Aut_S\,K^2$ is linear over  $K(t)$ for

\centerline{$S=\SL(2,\Z[t_1,t_1^{-1},\dots, t_n,t_n^{-1}][x_1,\dots x_m])$, and} 

\centerline{$K=\Q(t_1,t_1^{-1},\dots, t_n,t_n^{-1})((x_1,\dots x_m))$.}

\noindent This group is an example where the hypothesis ${\cal G}$ is fully used.

Next, we show that Theorem B is specific to dimension
$2$. For a finite codimensional ideal ${\bf m}$ in
$K[z,x,y]$ let $\Aut_{\bf m}\,K^3$ be the group of
all automorphisms of $K^3$ of the form

\centerline{ $(z,x,y)\mapsto 
(z+f,x+g,y+h)$,}

\noindent where $f,\,g$ and $h$ belongs to ${\bf m}$.
The group $\Aut_{\bf m}\,K^3$ is analogous to
$\Aut_1\,K^2$. However it is never linear, even if $K$ is finite, as shown by

\begin{MainD} The group $\Aut_{\bf m}\,K^3$ is not
linear, even over a ring.
\end{MainD}

The paper is organized as follows. In section 1,
the general notions concerning $\Aut\,K^2$ are defined,
and the classical result of van der Kulk \cite{vdK} 
is stated. In the next section, it is proved that
some linear groups obtained by amalgamation  are indeed
linear over a field. Therefore, it is possible to use
the theory of algebraic groups to show that some groups are not
linear. This is used to prove Theorem A.2 in section 3.

Borel and Tits \cite{BT} showed that certain group morphisms of
algebraic groups are semi-algebraic. A very simple form of their ideas are used to prove Theorem A.1, see Sections 4 and 5.

Next the proof of Theorem B and C is based on some  Ping-Pong ideas. These ideas were originally invented for the dynamic of groups with respect to 
the euclidean metric topologies \cite{FN}, but  they were used by Tits in the context of the ultrametric topologies \cite{T72}.

\bigskip\noindent
{\it Aknowledgements} J.P. Furter and R. Boutonnet informed us that they independently found a FG subgroup of $\Aut\, \Q^2$
which is not linear over a field \cite{BF}. We also heartily thank Y. Cornulier, 
S. Lamy and I. Soroko for interesting comments and E. Zelmanov for an inspiring talk.

 \section{The van der Kulk Theorem}

In this section, we review the basic facts about the amalgamated products of groups. Then the we recall the classical 
van der Kulk Theorem.

\bigskip
\noindent {\it 1.1 Group functors}

\noindent  A group functor  
is a functor $G:R \mapsto G(R)$ from the category of commutative  rings $R$ to the category of groups. In most cases, we will only define the group $G(K)$ for a
 field $K$ and the reader should understand that the definition over a ring is similar. When $I$ is an ideal of a ring $R$, we denote by
$G(I)$ the kernel of the map $G(R)\to G(R/I)$.
For the functorial approach to group theory, see 
e.g. \cite{D} and  \cite{T89}.

Nevertheless, will use a nonconsistent notation, namely 
$\Aut\, K^2$, 
for the group  of polynomial automorphisms of $K^2$, as well for its subgroups defined in the introduction.
Similarly, we will use the notation $\Aut\,K^3$ in section 10.

A group subfunctor $H\subset G$ of {\it finite codimension} means that the functor $R\mapsto G(R)/H(R)$ is represented 
by a scheme of finite dimension. The reader should understand this as an unformal definition,
and we will not care to be more precise. For example, 
the subgroups $\Aut_0\,K^2$   has codimension $2$   
in $\Aut\,K^2$.

\bigskip
\noindent {\it 1.2  Amalgamated products}
 
\noindent  Let $A$, $G_1$ and $G_2$ be groups. Their {\it free product} is denoted by $G_1 *G_2$, see \cite{MKS} ch.4.  Now let 
$A {\buildrel f_1\over\rightarrow} G_1$ and
$A {\buildrel f_2\over\rightarrow} G_2$ be two injective group morphisms and let $K\subset G_1 *G_2$ be the invariant subgroup
generated by the elements $f_1(a)f_2(a^{-1})$ when $a$ runs
over $A$. By definition the group $G_1*G_2/K$ is called 
the {\it amalgamated product} of $G$ and $H$ over $A$,
and it is denoted by $G_1 *_A G_2$, 
see e.g. \cite{S83}, ch. I. 
 In the literature, the amalgamated products are also called
free amalgamated products, see \cite {MKS} ch.8. 

Recall that the natural maps 
 $G_1\rightarrow\Gamma$ and $G_2\rightarrow\Gamma$ are
 injective, see the remark after the Theorem 1 of  ch. 1 in \cite{S83}. Hence, we will use a less formal terminology.
The group $A$ will be viewed as a common subgroup 
 of $G_1$ and $G_2$, and $G_1$ and $G_2$ will be viewed as
 subgroups of  $G_1 *_A G_2$.

  \bigskip \noindent 
{\it 1.3 Reduced words}

\noindent
The usual definition \cite{S83} of reduced words
is based on the right $A$-cosets. In order to avoid a confusion
between the set difference notation $A\setminus X$ 
and the $A$-orbits notation $A\backslash X$, we will
use a definition based on the left $A$-cosets.

Let $G_1$, $G_2$ be two groups sharing a common subgroup $A$,
and let $\Gamma=G_1 *_A G_2$.
Set $G_1^*=G_1\setminus A$,  $G_2^*=G_2\setminus A$
and let $T_1^*\subset G_1^* $ (respectively 
$T_2^*\subset G_2^* $) be a set of  representatives of  
$G_1^*/A$ (respectively of $G_2^*/A$).

Let $\Sigma$ be the set of all finite alternating sequences of ones and twos and let 
$\epsilon\in \Sigma$. 
A {\it reduced word} of {\it type $\epsilon$}  is a word
$(x_1,\dots x_n,x_0)$ where $x_0$ is in $A$, and
$x_i\in T_{\epsilon_i}^*$ for $i\geq 1$.
Let ${\cal R}$ be the set of all reduced words.
The next lemma is well-known, see  
 e.g. \cite{S83}, Theorem 1.

\begin{lemma}\label{words} The map

\centerline{ $(x_1,\dots x_n,x_0)\in {\cal R} 
 \mapsto  x_1\dots x_n x_0 \in G_1*_A G_2$}
 
 \noindent is bijective.
 
 \end{lemma}
 
 Set $\Gamma=G_1*_A G_2$. For $\gamma\in \Gamma\setminus A$
 there is some integer $n\geq 1$, some 
 $\epsilon=(\epsilon_1,\dots,\epsilon_n)\in \Sigma$ and some $g_i\in G_{\epsilon_i}^*$
 such that $\gamma=g_1...g_n$. It follows from
 \cite {S83} (see the remark after Theorem 1, ch. 1) that
 $\gamma=x_1...x_nx_0$ for some reduced word 
 $(x_1...x_nx_0)$ of type $\epsilon$. Since it is determined by $\gamma$, the sequence $\epsilon$ is called 
{\it the type} of $\gamma$.

 \bigskip \noindent 
{\it 1.4 Amalgamated product of subgroups}

\noindent
Let $G_1$, $G_2$ be two groups sharing a common subgroup $A$ and set $\Gamma=G_1*_A G_2$.
Let $G_1'\subset G_1$, $G_2'\subset G_2$
and $A'\subset A$ be subgroups with the property that $G_1'\cap A=G'_2\cap A=A'$.

 \begin{lemma}\label{subamal} (i) The natural map 
 $G_1'*_{A'} G_2'\to \Gamma$ is injective. 
 
 (ii) Let $\Gamma'\subset \Gamma$ be a subgroup
 and set $G_1'= G_1\cap \Gamma'$, 
 $G_2'= G_2\cap \Gamma'$ and $A'= A\cap \Gamma'$. 
Assume that $\Gamma'.A=\Gamma$. Then  we have
 
\centerline{$\Gamma'=G_1'*_{A'} G_2'$.}

 \end{lemma}

\begin{proof} {\it First step: proof of Assertion (i).}
For $i=1,2$ set $G_i^*=G_i\setminus A$,
$G_i^{'*}=G'_i\setminus A'$. Let 
$T^*_i\subset G_i^*$ and $T^{'*}_i\subset G_i^{'*}$
be a set of representatives of
$G_i^*/A$ and, respectively, of $G_i^{'*}/A'$.

Since the maps
$G_i'/A'\to G_i/A$ are injective, it can be assumed that  $T^{'*}_i\subset T^{*}_i$. Let $\cal R$ 
and $\cal R'$ be the set of reduced words of
$G_1*_A G_2$, and respectively of $G_1'*_{A'} G_2'$.
By definition, we have $\cal R'\subset \cal R$, thus
by lemma \ref{words} the map $G_1'*_{A'} G_2'\mapsto G_1*_A G_2$
is injective.

\noindent
{\it Second step: proof of Assertion (ii).}
We will use the notations of the previous proof.
Since $\Gamma'.A=\Gamma$, it follows that 
the  maps $G'_1/A'\to G/A$ and $G_2'/A'\to G_2/A$ are bijective. Therefore ${\cal R'}$ is the set of
all reduced words $(x_1,\dots,x_n,x_0)\in{\cal R}$ such that $x_0\in A'$. It follows easily that

\centerline{
$G_1*_A G_2/G_1'*_{A'} G_2'\simeq A/A'=\Gamma/\Gamma'$,}

\noindent and therefore we have 
$\Gamma'=G_1'*_{A'} G_2'$.

\end{proof}

 \bigskip\noindent
 {\it 1.5 The subgroup $\Elem(K)$ of elementary automorphisms of $K^2$}
 
\noindent By definition, an {\it elementary automorphism}  of $K^2$ is an automorphism $\phi\in \Aut\,K^2$ of the form 
 
 \centerline{$\phi:(x,y)\mapsto (z_1 x +t, z_2 y +f(x))$}
 
 \noindent for some $z_1, z_2\in K^*$, some $t\in K$ and some $f\in K[x]$. The group of elementary automorphism is denoted
 $\Elem (K)$.

 \bigskip\noindent
 {\it 1.6 The affine group $\Aff(2,K)$}

\noindent Let $\Aff(2,K)\subset \Aut\,K^2$ 
be the subgroup
of affine automorphisms of $K^2$.  Set 
$B(K)=\Aff(2,K)\cap \Elem(K)$ 
Indeed $B(K)$ is a Borel subgroup of 
$\Aff(2,K)$  and we have

\centerline{$\Aff(2,K)/B(K)\simeq \P^1_K$.}

\bigskip
\noindent
{\it 1.7 The van der Kulk Theorem}
 
\noindent Recall the classical  

\begin{vdKthm}\cite{vdK} We have 

\centerline{$\Aut\,K^2\simeq \Aff(2,K)*_{B(K)}\Elem(K)$.}

\end{vdKthm}

\section{Amalgamated Products and
Linearity.}

\noindent In this section, it is shown that, under a mild assumption, an amalgamated product $G_1*_A G_2$
which is linear over a ring is also linear over a field,
see Lemma \ref{criterion}.

\bigskip 
\noindent {\it 2.1 Linearity Properties}

\noindent
For a group $\Gamma$,  the strongest form of  linearity is the linearity over a field. On the opposite, there are also   groups 
$\Gamma$ which contain a FG  subgroup which is
not linear, even over a ring: these groups are nonlinear in the strongest sense.

\bigskip\noindent
{\it 2.2 Minimal embeddings}

\noindent Let $R$ be a commutative ring, let $n\geq 1$ and let $\Gamma$ be a subgroup of $\GL(n,R)$. For an ideal $J\subset R$, the elements $g\in \GL(n,J)$ can be written as $g=\id+A$, where $A$ is a $n$-by-$n$ matrix with entries in 
$J$.

The embedding $\Gamma\subset \GL(n,R)$ is called {\it minimal} if
for any ideal $J\neq 0$ we have $\Gamma\cap \GL(n,J)\neq \{1\}$.

\begin{lemma} \label{minimal} Let $\Gamma\subset \GL(n,R)$. There exist an ideal $J$ with
$\Gamma\cap \GL(n,J)=\{1\}$ such that the induced embedding
$\Gamma\to\GL(n,R/J)$ is minimal.
\end{lemma}

\begin{proof} Since $R$ is not necessarily noetherian,
the  proof requires Zorn Lemma.

Let $\cal S$ be the set of all ideals $J$
of $R$ such that $\Gamma\cap \GL(n,J)=\{1\}$. With respect to the inclusion, $\cal S$ is a poset.  For any chain $\cal C\subset\cal S$, the ideal 
$\cup_{I\in{\cal C}}\,I$ belongs to $\cal S$. Therefore Zorn Lemma implies that $\cal S$ contains a maximal element $J$. It follows that the induced embedding $\Gamma\to\GL(n,R/J)$ is minimal.
\end{proof}

\bigskip\noindent
{\it 2.3 Groups with trivial centralizers}

\noindent
By definition, a group $\Gamma$ 
{\it has trivial centralizers} if the centralizer of any nontrivial normal subgroup $K$ of $\Gamma$ is trivial. 
Equivalently, if $K_1$ and $K_2$  are commuting invariant subgroups of $\Gamma$, then one of them is trivial.

\begin{lemma}\label{linearity} Let  $\Gamma$ be a  group with trivial centralizers.

If $\Gamma$ is linear over a ring, 
then $\Gamma$ is also linear over a field.
\end{lemma} 

\begin{proof} By hypothesis we have
$\Gamma\subset \GL(n,R)$ for some commutative ring $R$.
By lemma \ref{minimal}, it can be assumed that the embedding 
$\Gamma\to \GL(n,R)$ is minimal. 

Let $I_1,\,I_2$ be ideals 
of $R$ with $I_1.I_2=0$. Let $A_1$ (respectively $A_2$) be
an arbitrary $n$-by-$n$ matrix with entries in $I_1$ 
(respectively in $I_2$) and set $g_1=1+A_1$, 
$g_2=1+A_2$. Since we have $A_1.A_2=A_2.A_1=0$, 
we have $g_1.g_2=g_2.g_1=1+A_1+A_2$, therefore 
$\GL(n,I_1)$ and $\GL(n,I_2)$ are commuting invariant subgroups
of $GL(n,R)$.

Since $K_1=\Gamma\cap GL(n,I_1)$ and  $K_2=\Gamma\cap GL(n,I_2)$
are commuting  invariant subgroups of $\Gamma$,  one of them is trivial. By minimality hypothesis, $I_1$ or $I_2$ is trivial.
Thus $R$ is prime.

It follows that $\Gamma\subset \GL(n,K)$, where 
$K$ is the fraction field of $R$.
\end{proof}

\bigskip\noindent
{\it 2.4 The hypothesis ${\cal H}$}

\noindent
Let  $G_1$, $G_2$ be two groups sharing a common subgroup $A$ and
 set $\Gamma=G_1*_A G_2$.

 Let $\Sigma$ be the set of all finite alternating sequences
 ${\bf \epsilon}=\epsilon_1,\dots,\epsilon_n$ of ones and twos.
For $i,\,j\in\{1,2\}$, let $\Sigma_{i,j}$ be
be the subset of all $\epsilon=(\epsilon_1,\dots,\epsilon_n)\in \Sigma$
starting with $i$ and ending with $j$ and 
let $\Gamma_{i,j}$ be the set of all $\gamma\in \Gamma$
of type $\epsilon$ for some $\epsilon\in \Sigma_{i,j}$.
Therefore we have
 
\centerline{ $\Gamma=A\sqcup \Gamma_{1,1}
 \sqcup \Gamma_{2,2}\sqcup \Gamma_{1,2}\sqcup \Gamma_{2,1}$.}

 By definition, the amalgamated product 
 $G_1*_A\,G_2$  is called {\it nontrivial} if
 $G_1\neq A$, or $G_2\neq A$. It is called 
{\it dihedral} if $G_1=G_2=\Z/2\Z$, and $A=\{1\}$,
and {\it nondihedral} otherwise.
 
 \noindent Let consider the following  hypothesis
 
 $({\cal H})$\hskip3mm For any $a\in A$ with $a\neq 1$, there is 
 $\gamma\in \Gamma$ such that
 $a^\gamma\notin A$
 
 \noindent where, as usual, $a^\gamma:=\gamma a \gamma^{-1}$.

 \begin{lemma}\label{conj} Let $\Gamma=G_1*_A\,G_2$ be a 
 nontrivial and nondihedral
 amalgamated product satisfying the hypothesis 
 ${\cal H}$.
 
 For any element $g\neq 1$ of $\Gamma$,
 there are 
 $\gamma_1,\,\gamma_2\in \Gamma$ such that
 
 \centerline {$g^{\gamma_1}\in \Gamma_{1,1}$ and
 $g^{\gamma_2}\in \Gamma_{2,2}$.}
 
\noindent In particular $\Gamma$ has trivial centralizers.
 \end{lemma}

 \begin{proof}

  First it should be noted that $A$ cannot be 
 simultaneously a subgroup of index 2 in $G_1$ and in $G_2$.
 Otherwise, $A$ would be an invariant subgroup of $G_1$ and $G_2$,
 the hypothesis  ${\cal H}$ would imply that $A=\{1\}$
 and $\Gamma$ would be the dihedral group.
  Hence it can be assumed that $G_2/A$ contains
 at least $3$ elements. 
 
 Next it is clear that
 $G^*_i.\Gamma_{j,k}\subset \Gamma_{i,k}$
  whenever $i\neq j$. Similarly, we have
 $\Gamma_{i,j}.G^*_k\subset \Gamma_{i,k}$
 whenever $j\neq k$.

 Now we prove the first claim, namely that the conjugacy class of any $g\neq 1$ intersects both $\Gamma_{1,1}$ and $\Gamma_{2,2}$.

By hypothesis, $G^*_1$ is not empty. Let 
 $\gamma_1\in G^*_1$. We have 
 $\Gamma_{2,2}^{\gamma_1}\subset\Gamma_{1,1}$. Simlarly we have
 $\Gamma_{1,1}^{\gamma_2}\subset\Gamma_{2,2}$ for some
 $\gamma_2\in G_2^*$. Therefore the claim is proved for
 any $g\in \Gamma_{1,1}\cup \Gamma_{2,2}$. Moreover
 it is now enough to prove that the conjugacy class of any
 $g\neq 1$ intersects 
 $\Gamma_{1,1}$ or $\Gamma_{2,2}$. 
 
 Assume now $g\in \Gamma_{2,1}$. We have $g=u.v$ for some
 $u\in G^*_2$ and $v\in\Gamma_{1,1}$. Since 
 $[G_2:A]\geq 3$, there is $\gamma\in G^*_2$ such that
 $\gamma.u$ belongs to $G^*_2$. It follows that
 $\gamma.g$ belongs to $\Gamma_{2,1}$, and therefore
 $g^\gamma$ belongs to $\Gamma_{2,2}$. 
 
 Next, for $g\in \Gamma_{1,2}$, it inverse
 $g^{-1}$ belongs to $\Gamma_{1,2}$ and
  the claim follows from the previous point.
 
 Last, for $g\in A$,  there is $\gamma\in \Gamma$ such that $g^\gamma$ is not in $A$, by hypothesis ${\cal H}$. Thus $g^\gamma$ belongs to $\Gamma_{i,j}$  for some
 $i,\,j$. So $g$ is conjugate to some element in
 $\Gamma_{1,1}\cup\Gamma_{2,2}$ by the previous considerations.
 
 Next we prove that $\Gamma$ has trivial centralizers.
 Let $K_1,\,K_2$ be nontrivial invariant subgroups. By the previous point, there are elements $g_1,\,g_2$ with
 
 \centerline{$g_1\in K_1\cap \Gamma_{1,1}$ and
 $g_2\in K_2\cap \Gamma_{2,2}$.}
 
\noindent Since we have $g_1 g_2\in \Gamma_{1,2}$ and $g_2 g_1\in \Gamma_{2,1}$, it follows that $g_1g_2\neq g_2g_1$. Therefore
$K_1$ and $K_2$ do not commute.
 
\end{proof}

 As an obvious corollary of Lemmas \ref{linearity} and
\ref{conj}, we get

\begin{lemma}\label{criterion} Let $\Gamma=G_1*_A\,G_2$ be a 
nontrivial amalgamated product satisfying the hypothesis ${\cal H}$.

If $\Gamma$ is linear over a ring, then it is linear over a field.
\end{lemma} 

\begin{proof} The dihedral group is linear over a field.
Hence it can be assumed that the amalgamated product 
$\Gamma=G_1*_A\,G_2$ is also nondihedral. Thus
the lemma is an obvious corollary of Lemmas \ref{linearity} and \ref{conj}.

\end{proof}

\section{A nonlinear FG subgroup of $\Aut_0\,\Q^2$.}

First we define a certain FG group $\Gamma=G_1*_A G_2$ and we will show that $\Gamma$ is nonlinear, even over a ring. Then we see that $\Gamma$ is
a subgroup of $\Aut_0\,\Q^2$ and therefore $\Aut_0\,K^2$ contains
many nonlinear FG subgroups for any characteristic zero field 
$K$, what proves Theorem A.2.

\bigskip
\noindent {\it 3.1 Quasi-unipotent endomorphisms}

\noindent Let $V$ be a finite dimensional vector space over an algebraically closed field $K$. An element $u\in \GL(V)$ is called
{\it quasi-unipotent} if all its eigenvalues are roots of unity.
The {\it quasi-order} of a quasi-unipotent endomorphism
$u$ is the smallest positiver integer $n$ such that $u^n$ is unipotent.

If $u$ is unipotent and $\ch\,K=0$, set

\centerline
{$\log\,u=\log(1-(1-u)):=\sum_{m\geq 1} \,(1-u)^m/m$,}

\noindent which is well-defined since $1-u$ is nilpotent.

\begin{lemma}\label{unipotent} Let $h, u\in \GL(V)$.
Assume that $u$ has infinite order and
$u^h=u^{2^k}$ for some $k\geq 1$.
Then  $u$ is quasi-unipotent of odd quasi-order $n$.

Moreover $K$  has characteristic zero and
$e^h=2^k e$, where $e=\log\,u^n$.

\end{lemma}

\begin{proof} Let $\Spec\,u$ be the spectrum of $u$.
By hypothesis the  map 
$\Spec\,u\to \Spec\,u, \lambda\mapsto\lambda^{2^k}$ is bijective,
hence for any $\lambda\in\Spec\,u$, we have 
$\lambda^{2^{kn}}=\lambda$ for some integer $n$. It follows that all 
eigenvalues are odd roots of unity, what proves that $u$ is
quasi-unipotent of odd quasi-order.

Over any field of finite characteristic, the unipotent
endomorphisms have finite order. Hence we have $\ch\,K=0$. The last assertion comes from the fact that
$\log$ induces a bijection from unipotent elements to
nilpotent elements.

\end{proof}

\noindent {\it 3.2 The Group $\Gamma=G_1*_A G_2$}

\noindent Set $G_1=\Z^2$, and let $\sigma, \sigma'$ be a basis $\Z^2$.
Set $\Z_{(2)}=\{x\in\Q\vert 2^n x\in\Z$ for $n>>0\}$.
Indeed $Z_{(2)}$ is the localization of the ring $\Z$ at $2$, but,
in what follows, we will only consider its group structure. The element $1$ in $Z_{(2)}$ will be denoted by $\tau$
and the addition in $Z_{(2)}$ will be  denoted multiplicatively.

Let $\sigma$ act on $\Z_{(2)}$ by multiplication by $2$, so we can consider the semi-direct product $G_2=\Z.\sigma\ltimes \Z_{(2)}$.
Let $\Gamma=G_1*_A G_2$, where $A=\Z\sigma$. It is easy to show
that $\Gamma$ is generated by $\sigma,\sigma'$ and $\tau$, and 
it is defined by the following two relations

\centerline{$\sigma\sigma'=\sigma'\sigma$, and
$\sigma \tau \sigma^{-1}=\tau^2$.}

\begin{lemma}\label{prepa1} The group
$\Gamma$ has trivial centralizers.
\end{lemma} 

\begin{proof} By Lemma \ref{criterion}, it is enough to
prove that the amalgamated product
$\Gamma=G_1*_A G_2$ satisfies the hypothesis ${\cal H}$.

Indeed we claim that, for any 
element $a\in A\setminus 1$, its conjugate 
$\tau^{-1} a \tau$ is not in $A$.
We have $a=\sigma^n$ for some $n\neq 0$. 

First assume that $n\geq 1$.
We have $\sigma^n \tau\sigma^{-n}=\tau^{2^n}$ and therefore
$\tau^{-1} a \tau=\tau^{2^n-1} a$ what proves that
$\tau^{-1} a \tau$ is not in $A$. Moreover
if $n\leq -1$, it follows that $\tau^{-1} a^{-1} \tau$ is not in $A$, therefore  $\tau^{-1} a \tau$ is not in $A$.

Hence the amalgamated product
$\Gamma=G_1*_A G_2$ satisfies hypothesis ${\cal H}$. 
\end{proof}

\noindent{\it 3.3 Nonlinearity of $\Gamma$}

\begin{lemma} \label{nonlinear}

The group $\Gamma$ is not linear, even over a ring.

\end{lemma}

\begin{proof} 
Assume that $\Gamma$ is linear over a ring.
By Lemma \ref{prepa1} the group $\Gamma$
has trivial centralizers. Thus, by Lemma \ref{linearity},
$\Gamma$ is linear over a field.

Let $\rho:\Gamma\to \GL(V)$ be an embedding, where $V$ is a finite dimensional vector space over a field $K$.
Since $\sigma \tau \sigma^{-1}=\tau^2$ it follows from 
Lemma \ref{unipotent} that $\rho(\tau)$ is quasi-unipotent and its
quasi-order $n$ is odd.  

Let $\Gamma'$ be the subgroup of $\Gamma$ generated by
$\sigma,\,\sigma'$ and $\tau^n$. Since the morphism
$\psi: \Gamma\to\Gamma'$ defined by $\psi(\sigma)=\sigma$,
$\psi(\sigma')=\sigma'$ and $\psi(\tau)=\tau^n$ is an isomorphism,
it can be can assumed that $\rho(\tau)$ is unipotent.

By Lemma \ref{unipotent},
$K$ has characteristic zero. Set
$u=\rho(\tau)$, $h=\rho(\sigma)$ and 
$e=\log\,u$.

By Lemma \ref{unipotent}, we have 

\centerline{$h e h^{-1}=2e$.}

It can be assumed that $K$ is algebraically closed.
Let $V=\oplus_{\lambda\in K}\, V_{(\lambda)}$ be the decomposition of $V$ into generalized eigenspaces relative to $h$. Also, for 
$\lambda\in K$, set
$V_{(\lambda)}^+=\oplus_{n\geq 0}V_{(2^n\lambda)}$. 

Let $\lambda\in K$. Since
$\rho(G_1)$ commutes with $h$ we have 

\centerline{$\rho(G_1)V_{(\lambda)}\subset V_{(\lambda)}$.}

\noindent  Moreover we have
$e.V_{(\lambda)}\subset V_{(2 \lambda)}$ and therefore

\centerline{$\rho(G_2)V_{(\lambda)}^+\subset V_{(\lambda)}^+$.}

\noindent It follows that $V_{(\lambda)}^+$ is a $\Gamma$-module.
Since $\tau$ acts trivially on $V_{(\lambda)}^+/V_{(2\lambda)}^+$,
the image of $\Gamma$ in 
$\GL(V_{(\lambda)}^+/V_{(2\lambda)}^+)$ is commutative. It follows that $V$ has a composition series by one-dimensional $\Gamma$-module, hence
the group $\Gamma$ is solvable.

However a solvable group  contains a nontrivial invariant abelian subgroup. 
This contradicts  that $\Gamma$ has trivial
centralizers by Lemma \ref{prepa1}. 
\end{proof}

\noindent{\it 3.4 Proof of Theorem A.2}

\begin{MainA2} Let $K$ be a field of characteristic zero. Then
$\Aut_0\,K^2$ contains FG subgroups which are not linear even over a ring.
\end{MainA2}

\begin{proof} Let define three polynomial automorphisms
$S, S'$ and $T$ of the plane as follows. First 
$S$ and $S'$ are linear automorphisms  where
$S=1/2\,\id$ and $S'$ is defined by the matrix
$\begin{pmatrix}
1 & 1 \\
1 & 0
\end{pmatrix}$.  Next $T$ is the quadratic automorphism
$(x,y)\mapsto (x, y+x^2)$. Set $K_1=<S,S'>$,
 $K_2=<S,T>$, $C=<S>$.

The eigenvalues of $S'$ are $1\pm\sqrt{5}\over 2$. Hence
$S'^n$ is not upper diagonal for $n\neq 0$. It follows that
 $K_1\cap B_0=C$. Moreover $K-2$ is the groups of automorphisms of the form
 
\centerline{ $(x,y)\mapsto (2^k x, ry)$,}

\noindent for $k\in\Z$ and  $r\in \Z_{(2)}$, therefore
 $K_2\cap B_0=C$. It follows from Lemma \ref{subamal}
 and  van der Kulk Theorem that
 $K_1*_C K_2$ is a subgroup of $\Aut\,K^2$. Moreover it is clear that $K_1*_C K_2\subset \Aut_0\,K^2$.
 
 There is a
 group isomorphism $\Gamma\to K_1*_C K_2$ sending
 $\sigma$ to $S$, $\sigma'$ to $S'$ and $\tau$ to $T$. 
 Therefore, by Lemma \ref{nonlinear}, the FG subgroup
$K_1*_C K_2\simeq \Gamma$ in $\Aut_0\,K^2$ is not linear, even over a ring.

\end{proof}

\section{Semi-algebraic characters.}

Let $K$ be an infinite field of characteristic $p$. It is shown in this section that the 
degree of any semi-algebraic character is well-defined,
see Lemma \ref{degree}.

\begin{lemma}\label{span} Let $n\geq 1$ be an integer prime to $p$. Then for any $x\in K$, there is an integer $m$ and a collection $x_1,\,x_2\dots x_m$ of elements of $K^*$ such that

\centerline{$x=\sum_{1\leq k\leq m}\,x_i^n$.}

\end{lemma}

\begin{proof} Let $K'$ be the additive span of
the set $\{y^n\vert y\in K\}$. Since 
$-y^n=(p-1) y^n$, $K'$ is an additive subgroup of $K$. Clearly,
$K'$ is a subring of $K$.
The formula $x^{-1}=(1/x)^n x^{n-1}$ for any
$x\in K\setminus 0$ implies  that $K'$ is a subfield.

We claim that $K'$ is infinite. Assume otherwise. Then
$K'=\F_q$ for some power $q$ of $p$. Since any element of $K$ is
algebraic over $K'$ of degree $\leq n$, it follows that
$K\subset \cup_{k\leq n} \F_{q^k}$, a fact that contradicts that $K$ is infinite.

It remains to prove that $K'=K$. 
For $x\in K$, let $P(t)\in K[t]$ be the polynomial $P(t)=(x+t)^n$. Let $y_0,\dots y_n$ be $n+1$
distinct elements in $K'$. We have $P(y_i)\in K'$ for any
$0\leq i\leq n$. Using Lagrange's interpolation polynomials, there exist  a polynomial
$Q(t)\in K'[t]$ of degree $\leq n$ such that $P(y_i)=Q(y_i)$ for any $0\leq i\leq n$. Since $P-Q$ has at least $n+1$ roots, we have
$P=Q$ and therefore $P(t)\in K'[t]$. We have

\centerline{$P(t)=t^n + nx t^{n-1}+\dots$,}

\noindent hence $nx$ belongs to $K'$, and therefore $x\in K'$, what proves that $K'=K$.

\end{proof}

\begin{lemma}\label{n=m} Let $K$ be an infinite field of characteristic $p$,
let $\mu:K\to K$ be a field automorphism and
let  $n,\,m$ be positive integers which are prime to $p$. Assume that

\centerline {$x^n=\mu(x)^m$, for all $x\in K^*$.}

Then we have $m=n$.

\end{lemma}

\noindent
{\it Remark:} Indeed,  the hypotheses of the lemma
imply also that $\mu=\id$, but this is not required in what follows. 

\begin{proof} 

{\it Step 1: proof for $K=\F_p(t)$.}
There are $a,\,b,\,c$ and $d\in \F_p$ 
with $ad-bc\neq 0$ such that
$\mu(t)=\frac{at+b}{ ct+d}$.  The identity

\centerline{$t^n=(\frac{at+b}{ct+d})^m$}

\noindent clearly implies that $n=m$.

\smallskip
\noindent {\it Step 2: proof for $K\subset \overline\F_p$.}
Since $K$ is infinite, $K$ contains arbitrarily large finite fields. So we have
$K\supset \F_{p^N}$ for some positive integer $N$  with $nm<p^N$. Since $K$ contains a unique
field of cardinality $p^N$ we have $\mu(\F_{p^N})=\F_{p^N}$.
There is an non negative integer $a<N$ such that
$\mu(x)= x^{p^a}$ for any $x\in \F_{p^N}$. Therefore we have

\centerline{$x^n=x^{m{p^a}}$ for all $x\in F_{p^N}^*$.}

\noindent It follows that $n\equiv m{p^a}\, \mod p^N-1$. 
Let $n=\sum_{k\geq 0}\, n_k p^k$ and
$m=\sum_{k\geq 0}\, m_k p^k$ be the $p$-adic expansions of $n$ and $m$. By definition each  digit $n_k$, $m_k$ is an integer between $0$ and $p-1$. Since $n.m<p^N$, we have 
$m_k=m_k=0$ for $k\geq N$.

For each integer $k$, let $[k]$ be its residue modulo $N$, so
we have $0\leq [k]<N$ by definition. We have

\centerline{$n\equiv mp^a\equiv \sum_{0\leq k<N} m_k p^{[a+k]}\, \mod p^N-1$.}

\noindent 
Since both  integers $n$ and $\sum_{0\leq k<N} m_k p^{[a+k]}$ 
belongs to $[1,p^N-1]$ and are congruent modulo $p^N-1$, it
follows that

\centerline{$n= \sum_{0\leq k<N} m_k p^{[a+k]}$.}

 Assume $a>0$. 
We have $n_a=m_0$ and $n_0=m_{N-a}$.
Since $n$ and $m$ are prime to $p$, the digits $n_0$ and
$m_0$ are not zero, therefore we have
$n_a\neq 0$ and $m_{N-a}\neq 0$. It follows that

\centerline{$n\geq p^a$ and $m\geq p^{N-a}$,}

\noindent which contradicts  that $nm<p^N$.

Therefore, we have $a=0$, and the equality $n=m$ is obvious.

\smallskip
\noindent{\it Step 3: proof for any infinite field $K$.} 
In view of the previous step, it can be assumed that
$K$ contains a transcendental element $t$. Therefore,
$K$ contains the subfield $L:=\F_p(t)$.

By Lemma \ref{span}, $L$ and $\mu(L)$ are the additive span of the set 
$\{x^n\vert x\in K^*\}=\{\mu(x)^m\vert x\in K^*\}$. Therefore,
we have $\mu(L)=L$. Since $\mu$ induces a field automorphism of $L$,
the equality $n=m$ follows from the first step.

\end{proof}

Let $L$ be another field. A group morphism 
$\chi:K^*\to L^*$ is called a {\it character} of $K^*$.
Let $X(K^*)$ be the set of all $L$-valued
characters of $K^*$. 
Let $n$ be an integer prime to $p$. A character
$\chi\in X(K^*)$ is called {\it semi-algebraic} of {\it degree}
$n$ if 

\centerline{$\chi(x)=\mu(x)^n$, for any $x\in K^*$,}

\noindent for some  field embedding $\mu: K\to L$.

Let ${\cal X}_n(K^*)$  be the set of all semi-algebraic characters of $K^*$ of degree $n$. Of course it can be assumed that $L$ has characteristic $p$, otherwise the set ${\cal X}_n(K^*)$ is empty. 

The next lemma shows that for an integer $n>0$ prime to $p$, the degree 
is uniquely determined by the character $\chi$. Indeed this statement is also true for the
negative integers $n$ and moreover the  field embedding $\mu$ is also determined by $\chi$. For simplicity of the exposition, the lemma is stated and proved in its minimal form. 

However, it should be noted that the condition that $n$ is prime to $p$ is essential. 
Indeed, if $\mu:K\to L$ is a field embedding into a perfect field $L$, then
$\mu(x)^n=\mu'(x)^{pn}$, where $\mu'$ is the field embedding
defined by $\mu'(x)=\mu(x)^{1/p}$.

\begin{lemma}\label{degree} Let $n\neq m$ be positive integers which are prime to $p$. Then ${\cal X}_n(K^*)\cap {\cal X}_m(K^*)=\emptyset$.
\end{lemma}

\begin{proof} Let $n$, $m$ be positive integers prime to $p$
and let $\chi\in {\cal X}_n(K^*)\cap {\cal X}_m(K^*)$.
By definition, there are fields embeddings 
$\nu, \nu':K\to L$ such that

\centerline{$\chi(x)=\nu(x)^n=\nu'(x)^m$ for any $x\in K^*$.}

\noindent By Lemma \ref{span}, $\nu(K)$ and $\nu'(K)$ are the linear span of $\Image\,\chi$. Therefore
we have $\nu(K)=\nu'(K)$ and 
$\mu:=\nu'^{-1}\circ\nu$ is a well defined field automorphism of 
$K$. We have

\centerline{
$\mu(x)^n=\nu'^{-1}\circ\nu(x^n)= \nu'^{-1}\circ\nu'(x^m)=x^m$.}

Therefore by Lemma \ref{n=m}, we have $n=m$.

\end{proof}

\section{Nonlinearity of  $\SAut_0\,K^2$ for $K$ infinite }

This section provides the proof of Theorem A.1. 
Except stated otherwise, it will be assumed that $K$ has characteristic $p\neq 0$ throughout the whole section.
We follow the same idea than \cite{BT}, namely that some  abstract morphisms of algebraic groups are, somehow, semi-algebraic.

\bigskip\noindent
{\it 5.1 Infinite rank elementary $p$-groups}

\noindent
Recall that an {\it elementary abelian $p$-group} is simply a
$\F_p$-vector space $E$ viewed as a group. Its $\F_p$-dimension
is called the {\it rank} of $E$.

\begin{lemma}\label{inf-rank} Let $L$ be a field and let $E$ be
an elementary abelian $p$-group of infinite rank.
If $\ch\,L\neq p$, then $E$ is not linear over $L$.
\end{lemma}

\begin{proof} It can be assumed that $L$ is algebraically closed.
Let $V$ be a  vector space over $L$ of dimension $n$ and let
$F\subset \GL(V)$ be an elementary abelian $p$-subgroup. For any character $\chi: K\to L^*$, let $L_\chi$ be the 
corresponding one-dimensional representation of $F$. 

Since $\ch\,L\neq p$, $F$ is a diagonalizable subgroup of $\GL(V)$.
So there is an isomorphism of $F$-modules 
 
 \centerline{$V\simeq \oplus_{1\leq i\leq n}\,L\,_{\chi_i}$.}
 
\noindent Since $\cap \,\Ker\,\chi_i$ has $\F_p$-codimension
 $\leq n$, it follows that the rank of $F$ is $\leq n$.

 Therefore no infinite rank elementary abelian $p$-group
 is  linear over $L$.

\end{proof}

\bigskip\noindent
{\it 5.2 Recognition of semi-algebraic characters}

\noindent From now on, $L$ is an algebraically closed field of characteric $p$.  Roughly speaking Lemma
\ref{char} states that any abstract embedding
$G_n(K)\subset \GL(N,L)$ is enough to recognize the field $K$
and the integer $n$, if $n$ is prime to $p$.

\begin{lemma}\label{reco} Let $\chi\in X(K^*)$, let $\mu:K\to L$
be a nonzero additive map and let $n$ be a positive
integer prime to $p$.

Assume that

\centerline{$\mu (x^ny)=\chi(x)\mu(y)$}

\noindent for any $x\in K^*$ and $y\in K$. 

Then
$\chi$ is a semi-algebraic character of degree $n$.

\end{lemma}

\begin{proof} By Lemma \ref{span}, $K$ is the additive span
of $(K^*)^n$. Therefore there is some $x\in K^*$ such that
$\mu(x^n)\neq 0$. Since $\mu(x^n)=\chi(x)\mu(1)$, it follows that
$\mu(1)\neq 0$. After rescaling $\mu$, it can be assumed that
$\mu(1)=1$.

Let $x,\,y \in K^*$. We have $\mu(x^n)=\chi(x)\mu(1)=\chi(x)$,
and therefore $\mu(x^n y^n)=\chi(x)\mu(y^n)=\mu(x^n)\mu(y^n)$.
By lemma \ref{span},  $K$ is the additive span
of $(K^*)^n$. It follows that

\centerline{$\mu(ab)=\mu(a)\mu(b)$, for any $a,\,b\in K$.}

\noindent  Hence $\mu$ is a field embedding.
Moreover, we have $\chi(x)=\mu(x^n)=\mu(x)^n$ for any $x\in K^*$. Therefore
$\chi$ is a  semi-algebraic character of degree $n$.
\end{proof}

For $n\geq 1$, let $G_n(K)$ be the semi-direct product 
$K^*\ltimes K$, where any $z\in K^*$ acts on $K$ as $z^n$.
More explicitly, the elements of $G_n(K)$ are denoted 
$(z,a)$, with $z\in K^*$ and $a\in K$ and the product
is defined by
   
 \centerline{  $(z,a).(z',a')=(zz',z'^na+a')$,}
 
 \noindent for any $z,\,z'\in K^*$ and $a,\,a'\in K$.

Let $V$ be a finite dimensional $L$-vector space and
let $\rho:G_n(K)\to \GL(V)$ be a group morphism. With respect to 
the subgroup  $(K^*,0)$ of $G_n(K)$, there is a 
decomposition of $V$ as 

\centerline{$V=\oplus_{\chi\in X(K^*)}\, V_{(\chi)}$}

\noindent where 
$V_{(\chi)}:=\{v\in V\vert \forall x\in K^* :(\rho(x,0)-\chi(x))^n v=0$
for $n>>0\}$ is the generalized eigenspace associated with the character $\chi$.
Here and in the sequel, $\rho(z,a)$ stands for $\rho((z,a))$, 
for any $(z,a)\in G_n(K)$.

\begin{lemma}\label{char} Let $n$ be a positive integer prime to $p$.
$\rho:G_n(K)\to \GL(V)$ be an injective morphism. Then we have

\centerline{$\End (V)_{(\chi)}\neq 0$}

\noindent for some $\chi\in {\cal X}_n(K^*)$.

\end{lemma}

\begin{proof}
Set

\centerline{$V_0=\{v\in V\vert \rho(1,a)v=v,\, \forall a \in K\}$, and}

\centerline{
$V_1=\{v\in V\vert \rho(1,a)v=v \,\mod V_0 ,\,\forall a \in K\}$.}

\noindent In what follows, $K$ and $K^*$ denote the subgroups
$\{1\}\ltimes K$ and $K^*\ltimes\{0\}$ of $G_n(K)$.
Since  $\rho(1,a)^p=0$ for any $a\in K$, the $K$-module
$V$ is unipotent.  It follows that

\centerline{$V_1\supsetneq V_0\neq 0$.}

Clearly, $V_1$ is a $G_n(K)$ submodule, and let $\rho_1$
be the restriction of $\rho$ to $V_1$.
Let $\theta: K\rightarrow \End(V_1)$ be the map defined by
 $\theta(a)=\rho_1(1,a)-1$ for $a\in K$.
 By definition, we have  $\theta(a)(V_1)\subset V_0$ and
 $\theta(a)(V_0)=\{0\}$. Hence we have 
  $\theta(a)\circ\theta(b)=0$ for any $a,\,b\in K$. It follows easily that 
  
  \centerline{$\theta(a+b)=\theta(a)+\theta(b)$.}
  
\noindent The identity 
$(z,0).(1,a).(z,0)^{-1}=(1,z^na)$ in $G_n(K)$ implies that

\centerline{
$\rho_1(z,0)\circ\theta(a)\circ 
\rho_1(z,0)^{-1}=\theta(z^n a)$, for any $z\in K^*$ and $a\in K$.}

Let $W\subset \End(V_1)$ be the $G_n(K)$-module generated by $\Image\,\theta$. Since $\theta\neq 0$
 there is a non-zero $K^*$-equivariant map 
 $g:W\to L_\chi$, where $\chi$ is a character of 
 $K^*$ and $L_\chi$ is the corresponding one-dimensional 
 $K^*$-module.

Set $\mu=g\circ \theta$. It follows from the previous identities
that $\mu$ is additive and 

\centerline{$\mu(z^n a)=\chi(x) \mu(a)$,}

\noindent for any $z\in K^*$ and $a\in K$. 
By Lemma \ref{reco}, the character $\chi$ is 
semi-algebraic of degree $n$. Since $L_\chi$ is a subquotient of $\End(V)$, it follows that
$\End(V)_{(\chi)}\neq 0$.

\end{proof}

\bigskip\noindent
{\it 5.3 Nonlinearity over a field of $\SElem_0(K^2)$}

\noindent
Set
 
 \centerline{$\SElem_0 (K)=\Elem (K)\cap \SAut_0\,K^2$.}
 
\noindent For  $n\geq 2$,  we will identify
$G_n(K)$ with the subgroup of  $\SElem_0(K)$ of all
automorphism of the form

\centerline{$(x,y)\mapsto (zx, z^{-1}y +a x^{n-1})$}

\noindent for some $z\in K^*$ and $a\in K$.

 \begin{lemma}\label{NLp} Let $K$ be an infinite field of characteristic $p$. The group $\SElem_0(K)$ is not linear over a field.
 \end{lemma}

 \begin{proof} Assume otherwise, and let
 $\rho: \SElem_0(K)\to \GL(V)$, where $V$ is a finite dimensional vector space over an algebraically closed field $L$. By Lemma \ref{inf-rank}, the field
 $L$ has characteristic $p$.

 Let $\Omega$ be the set of all characters $\chi$ of
 $K^*$, such that $\End(V)_{(\chi)}\neq 0$.
 Let $n\geq 2$ prime to $p$.
 Since $\SElem_0(K)\to \GL(V)$ contains
 $G_n(K)$, $\Omega$ contains some
 character in $\chi_n\in {\cal X}_n(K^*)$  by Lemma \ref{char}.
  By Lemma \ref{degree} these characters 
  $\chi_n$ are all
 distinct, which contradicts that $\Omega$ is a finite set.
 
 \end{proof}
 
  \begin{lemma}\label{NL0} The group $\SElem_0(\Q)$ is not linear over a field.
 \end{lemma}

 \begin{proof} Assume otherwise, and let
 $\rho: \SElem_0(\Q)\to \GL(V)$, where $V$ is a finite dimensional vector space over a field $L$. 
 Let $h$ be the automorphism $(x,y)\mapsto (2x,1/2 y)$.

 For any $n\geq 1$, let $u_n$ be the automorphism
 $(x,y)\mapsto (x, y+x^n)$. We have 
 $u_n^h=u_n^{2^{n+1}}$. By Lemma \ref{unipotent}, 
 $\rho(u_n)$ is quasi-unipotent and let
 $m_n$ be its quasi-order.  Moreover $L$ has characteristic zero. Set
 $e_n=\log \,\rho(u_n^{m_n})$. We have
 $e_n^{H}=2^{n+1}e_n$, where $H=\rho(h)$. This would imply that $\Ad(H)$ has infinitely many eigenvalues.

 \end{proof}

 \bigskip\noindent
{\it 5.4 Nonlinearity over a ring of $\SAut_0\,K^2$}

From now on, $K$ is a field of arbitrary characteristic.

Set 
$SB_0(K)=B(K)\cap \SAut_0\,K^2=B(K)\cap \SL(2,K)$.
 
 \begin{lemma}\label{H} The amalgamated product
 $\SL(2,K)*_{SB_0(K)}\,\SElem_0(K)$ satisfies hypothesis
 ${\cal H}$.
 \end{lemma}

 \begin{proof} Let $g\in SB_0(K)$ with $g\neq 1$.
 
 First, if $g$ is not an homothety,
 there is $\gamma\in \SL(2,K)$ such that $g^\gamma$ is not upper triangular and therefore we have $g^\gamma\notin SB_0(K)$. 
 
 Otherwise $g=-1$ and $\ch\,K\neq 2$. Let $\gamma$ be the automorphism $(x,y)\mapsto (x,y+x^2)$. 
Then  $g^\gamma$ is the automorphism

\centerline{$(x,y)\mapsto (- x,- y +2 x^2)$,}

\noindent so $g^\gamma$  is not in $SB_0(K)$. 
 
\end{proof}

 \begin{MainA1} If $K$ be an infinite field of
 arbitrary characteristic. The group 
 $\SAut_0\,K^2$ is not linear, even over a ring.
\end{MainA1}

\begin{proof}  
 
\noindent By Lemmas \ref{NLp} and \ref{NL0}, the group 
$\SElem_0(K^2)$ is not linear over a field. Set
$\Gamma=\SAut_0\,K^2\simeq\SL(2,K)*_{SB_0(K)}\,\SElem_0(K)$.  By Lemma \ref{H}
this amalgamated product 
satisfies hypothesis ${\cal H}$. Thus
it follows from  Lemma  \ref{linearity}
that $\Gamma$ is not linear, even over a ring.

Using van der Kulk Theorem and Lemma \ref{subamal}, we have 
$\SAut_0\,K^2\simeq\Gamma$. Therefore
$\SAut_0\,K^2$ is not linear, even over a ring.
 
\end{proof}

\noindent{\it 5.5 Comparison with Cornulier's Theorem}

\noindent Let   $G_{\mathrm{Cor}}$ be the
group of all automorphisms of $\Q^2$ of the form

\centerline{$(x,y)\mapsto (x+u, y+f(x))$,}

\noindent where $u\in \Q$ and $f(t)\in \Q[t]$. The group
$G_{\mathrm{Cor}}$, which is not FG, has been used in \cite{C} to prove

\begin{Cornu} The group $G_{\mathrm{Cor}}$ is not linear over a field. Consequently, the group $\Aut\,\Q^2$ is not linear over a field.

\end{Cornu}

\noindent Indeed the proof is based on the fact that 
$G_{\mathrm{Cor}}$ is nil but not nilpotent.

Here we used the subgroup $\SElem_0(K)$, which
is not linear over a field. However, it should be noted that both subgroups are linear over a ring.
Indeed let $R=\Q[[x]]\oplus \Q((x))/\Q[[x]]$, where $\Q((x))/\Q[[x]]$ is a  square-zero ideal and let $R':=K^\infty$ be an infinite product of $K$. Then
the group $G_{Cor}$ is linear over the ring $R$, and the group $\SElem_0(K)$ is linear over $R'$. Since we will not use this fact, the proof is left to the reader.

 \section{The Tits Ping-Pong}

\bigskip
\noindent
{\it 6.1 The Ping-Pong lemma}

\noindent
Let $(E_p)_{p\in P}$ be a collection of groups indexed by
a  set $P$. Let $\Gamma:=*_{p\in P}\,E_p$ be the free product of
these groups. Let $\Sigma_P$ be the set of all finite sequences
$(p_1,\dots,p_n)$ such that $p_i\neq p_{i+1}$ for any
$i<n$. For each $p\in P$, set 
$E^*_p=E\setminus \{1\}$.

Let $\gamma\in \Gamma$. There is a unique
${\bf p}=(p_1,\dots,p_n)\in \Sigma_P$ and a unique decomposition  
of $\gamma$

\centerline{$\gamma=\gamma_1\dots\gamma_n$,} 

\noindent
where $\gamma_i\in E^*_{p_i}$. The sequence ${\bf p}$ is
called the {\it type} of $\gamma$.

The free product $\Gamma:=*_{p\in P}\,E_p$ is called {\it nontrivial} if 

(i) for any $p\in P$, $E_p\neq\,\{1\}$,

(ii) $\Card\,P\geq 2$.

\noindent Also recall that the free product 
$*_{p\in P}\,E_p$ is called {\it nondihedral} if it is not the  free product
$\Z/2\Z*\Z/2\Z$.

\begin{lemma}\label{Ping-Pong} Assume that a nontrivial
nondihedral 
free product $\Gamma=*_{p\in P}\,E$ acts on some set $\Omega$. Let $(\Omega_p)_{p\in P}$ be a collection of subsets 
in $\Omega$. Assume

(i) The subsets $\Omega_p$ are nonempty and disjoint, and

(ii) we have $E^*_p.\Omega_q\subset \Omega_p$ whenever $p\neq q$.

Then the action of $\Gamma$ on $\Omega$ is faithful.
\end{lemma}

\noindent The hypothesis (ii) is called a Ping-Pong hypothesis.

\begin{proof} Let $\gamma\in\Gamma$ with $\gamma\neq 1$,
and let ${\bf p}=(p_1,\dots, p_n)$ be its type. 
We claim that there is $x\in \Omega$ such that $\gamma.x\neq x$

If $\Card\,P\geq 3$ let $q\in P$ with 
$q\neq p_1$ and $q\neq p_n$. Let $x\in \Omega_q$.
By the Ping-Pong hypothesis,  $\gamma.x$ belongs to
$\Omega_{p_1}$, therefore $\gamma.x\neq x$. 

If $\Card\,P=2$ it can be assumed that $P=\{1,2\}$. It follows from Lemma \ref{conj} that $\gamma$ has a conjugate $\gamma'$ of
type ${\bf q}=(q_1,\dots,q_m)$ with $q_1=q_m=1$. Let 
$x'\in \Omega_2$. By the Ping-Pong hypothesis,  $\gamma'.x'$ belongs to
$\Omega_{1}$, therefore $\gamma'.x\neq x'$. It follows that there is some $x\in\Omega$ with $\gamma.x\neq x$.

In both cases, any $\gamma\neq 1$ acts nontrivially.
Hence  the action of $\Gamma$ on $\Omega$ is faithful.
 \end{proof}

\noindent {\it 6.2 Mixture of free products, amalgamated products and semi-direct product}

\noindent This section is devoted to three technical lemmas,
for groups with a mix of free products, amalgamated products and
semi-direct products.

Let $G$ be a group. A {\it $G$-structure} on a group $E$ is a $G$-action on $E$, 
where $G$ acts by  group automorphisms. Equivalently, it means that the semi-direct product 
$G\ltimes E$ is  defined. For simplicity, a group $E$ with a $G$-structure is called a 
{\it $G$-group}. Two $G$-groups $E$, $E'$  are called
{\it $G$-isomorphic} if there is an isomorphism from $E$ to $E'$ which commutes with the
$G$-structure.

Let $P$ be a set on which $G$ acts, let $E$ be a group and 
for each $p\in P$ let $E_p$ be a copy of $E$. A $G$-structure on $*_{p\in P}\, E_p$ is called
{\it compatible with the $G$-action on $P$} if $E_p^g=E_{g(p)}$ for any $g\in G$ and $p\in P$. Roughly speaking, 
it means that $G$ acts on  $*_{p\in P}\, E_p$ by permuting the free factors. When it is not ambiguous, we just say that the 
$G$-structure is {\it compatible}. 

\noindent 

Now let $G$, $U$ be two groups sharing a common group
$A$. Assume moreover that $U=A\ltimes E$,
for some normal subgroup $E$ of $U$. Set
$\Gamma_0=G*_{A}U$. The natural map 
$U\to U/E\simeq A$ induces a group morphism

\centerline {$\chi: \Gamma_0=G*_{A}U\to G=G*_A A$.}

\noindent and let $\Gamma_1$ be its kernel.

Set $P=G/A$. For $\gamma\in G$, the subgroup
$E^\gamma$ lies obviously in $\Gamma_1$ and it depends only on
$\gamma\mod A$.

\begin{lemma}\label{mixing1} The group $\Gamma_1$ is the free product
of all $E^\gamma$, where $\gamma$ runs over $P$.
\end{lemma} 

\begin{lemma}\label{mixing2} For each $p\in P$, let $E_p$ be
a copy of $E$. A compatible $G$-structure on $*_{p\in P} E_p$
obviously provides a $A$-structure on $E$, and we have

\centerline{$G\ltimes (*_{p\in P} E_p)\simeq 
G*_{A} (A\ltimes E)$.}
\end{lemma} 

\begin{proof} Set $\Gamma'_1=*_{\gamma\in P}\, E^\gamma$.
Clearly $G$ acts over $\Gamma'_1$, so we can consider

\centerline{$\Gamma'_0:=G\ltimes \Gamma'_1$.}

\noindent Using the universal properties of amalgamated, free and
semi-direct products, one defines morphisms
$\phi:\Gamma_0\to \Gamma'_0$ and $\psi:\Gamma'_0\to \Gamma_0$
which are inverse of each other. It follows that 
$\Gamma_0$ and $\Gamma'_0$ are isomorphic, and 
$\phi$ induces an isomorphism from $\Gamma_1$
to $\Gamma_1'$, which proves Lemma \ref{mixing1}.

Let $1$ be the distinguished point of
$G/A$. We have $E_1^a=E_1$ for any $a\in A$, hence the group
$E=E_1$ has an $A$-structure. The rest of the proof 
of Lemma \ref{mixing2} follows from
universal properties, as before.

\end{proof}

For the last lemma, let $G$ be a group acting on a set $P$
and  let $E$ and $E'$ be two other groups. For each $p\in P$, let $G_p$ be the stabilizer in $G$ of $p$ and let
$E_p$ (respectively $E'_p$) be a copy of $E$ (respectively of
$E'$). 

Assume given some compatible $G$-structures on 
$*_{p\in P}\,E_p$ and $*_{p\in P}\,E'_p$. Obviously, it
provides some $G_p$-structure on  $E_p$ and $E'_p$, for any $p\in P$.

\begin{lemma}\label{mixing3} Assume that the groups $E_p$ and $E'_p$ are $G_p$-isomorphic 
for any $p\in P$. Then the groups $G\ltimes(*_{p\in P}\,E_p)$ and 
$G\ltimes(  *_{p\in P}\,E'_p)$ are isomorphic.
\end{lemma}

\begin{proof} First, assume that $G$ acts transitively on 
$P$. Let $p$ be a point of $P$, 
and let $A:=G_p$ be its stabilizer. It follows from Lemma \ref{mixing2} that

\centerline{$G\ltimes(*_{p\in P}\,E_p)   \simeq G*_{A}
(A\ltimes E_p))$, and}

\centerline{$G\ltimes(*_{p\in P}\,E'_p)\simeq G*_{A}
(A\ltimes E'_p)$.}

\noindent Hence  the groups 
$*_{p\in P}\,E_p$ and $*_{p\in P}\,E'_p$ are $G$-isomorphic. 
From this,  it follows that the groups 
$*_{p\in P}\,E_p$ and $*_{p\in P}\,E'_p$ are 
$G$-isomorphic even if $G$ does not act transitively on $P$. Therefore $G\ltimes(*_{p\in P}\,E_p)$ and 
$G\ltimes(  *_{p\in P}\,E'_p)$ are isomorphic.
\end{proof}

\bigskip
\noindent
{\it 6.3 The subgroups of $\GL_S(2,K[t])$ in $\GL(2,K[t])$}

\noindent For $G(t)\in \GL(2,K[t])$, let $G(0)$ its evaluation at $0$. For any subgroup $S\subset \GL(2,K)$
set

\centerline
{$\GL_S(2,K[t]):=\{G(t)\in \GL(2,K[t])\vert G(0)\in S\}$.}

\noindent For $S=\{1\}$, the  group $\GL_S(2,K[t])$ will be denoted by
$\GL_1(2,K[t])$.

For any $\gamma\in\P^1_K$, let $e_\gamma\in\End(K^2)$ such that
$e_\gamma^2=0$ and $\Image e_\gamma=\gamma$. Since
$e_\gamma$ is unique up to a constant multiple, the group

\centerline{$E_\gamma:=\{ \id + tf(t) e_\gamma \vert\,f\in K[t]\}$}

\noindent is a well defined subgroup of $\GL_1(2,K[t])$.

\begin{lemma}\label{generation} The group $\GL_1(2,K[t])$ is generated by its subgroups $E_\gamma$, where $\gamma$ runs over $\P^1_K$.
\end{lemma}

\begin{proof} 

For $A$, $B\in\End(K^2)$, set

\centerline{$<A\vert \,B>=\det\,(A+B) -\det\,A-\det\,B$.}

\noindent Any element $G(t)\in \GL_1(2,K[t])$ can be written as a
polynomial

\centerline{$G(t)=\sum_{n\geq 0}\, A_n t^n$}

\noindent where $A_n$ belong to $\End(K^2)$ and $A_0=\id$. 

The proof runs by induction on the degree $N$ of $G$. 
It can be assumed that $N\geq 1$. For any $n\geq 0$, let
$I_n$ be the set of pairs of integers $(i,j)$
with $0\leq i<j$ and $i+j=n$. We have
$\det\,G(t)=\sum_{n\geq 0}\, c_n t^n$, where the scalars $c_n$ are given by

\centerline{$c_n=\det\,A_{n/2} + \sum_{(i,j)\in I_n}\, <A_i\vert A_j>$,}

\noindent where it is understood that  $\det\,A_{n/2}=0$ if $n$ is odd.

Since $\det\,G(t)$ is an invertible polynomial, we have $c_n=0$ for any $n>0$.
The identity $c_{2N}=0$ implies that $\det\,A_N=0$ hence $A_N$ has rank one.
Let $\delta$ be the image of $\,A_N$ and 
set $E=\{a\in\End(K^2)\vert\,\Image\,a\subset\delta\}$

There exist an integer $n$ such that

\centerline{$A_n\notin E$,
but $A_m\in E$ for any $m>n$.}

We have

\centerline{$0=c_{N+n}=\det\,A_{N+n\over 2} +\sum_{(i,j)\in I_{N+n}}\,<A_i\vert A_j>$}

\noindent Since $E$ consists of rank-one endomorphisms, we have 
$\det a =0$ for any $a\in E$ and $<a,b>=0$ for any $a,\,b\in E$.
Therefore we have $\det\,A_{N+n\over 2}=0$ and $<A_i\vert A_j>=0$, whenever
$(i,j)$ lies in $I_{N+n}$ and $i\neq n$. Thus it follows that

\centerline{$<A_n\vert \,A_N>=0$.}

Set $\delta'=\Ker A_N$.  An easy computation shows that the previous relation 
$<A_N\vert \,A_n>=0$ implies
that $A_n(\delta')\subset \delta$. Set $B=e_\delta\circ A_n$. 
Clearly we have $\Ker \,B\supset\delta'$ and $\Image\,B\subset\delta$, therefore
$B$ is proportional to $A_N$. Since $A_n$ is not in $E$, we have $B\neq 0$ and therefore
$ce_\delta\circ A_n=A_N$ for some $c\in K$.
Set 

\centerline{$H(t)= (1-ct^{N-n} e_\delta).G(t)$.}

\noindent We have $e_{\delta}.a=0$ for any $a\in E$, therefore we have
$e_\delta.A_m=0$ for any $m>n$. It follows that 
$H(t)$ has degree $\leq N$. Moreover, its degree $N$ component
is $A_N-ce_\gamma.A_n=0$. Therefore $H(t)$ has degree
$<N$, and the proof runs by induction.
\end{proof}

\bigskip
\noindent
{\it 6.4 Free products in $\GL_1(K[t])$}

\noindent Let $K$ be a field. 

\begin{lemma}\label{Tits}  We have 

\centerline{$\GL_1(2,K[t])= *_{\delta\in\P^1_K}\,E_\delta$.} 
\end{lemma}

\begin{proof} Set $\Omega:=K[t]^2\setminus\{0\}$. Any $v\in\Omega$
can be written as a finite sum $v=\sum_{0\leq k}\,v_k\otimes t^k$,
where  $v_k$ lies in $K^2$. The biggest integer $n$ with 
$v_n\neq 0$ is the {\it degree} of $v$ and 
$\hc(v):=v_n$ is called its highest component.
For any $\delta\in \P^1_K$, set 
$E^*_\delta=E_\delta\setminus\{1\}$ and

\centerline
{$\Omega_\delta=\{v\in \Omega\vert\,\, \hc(v)\in\delta\}$.}

Let $H(t)\in E^*_\delta$. Since $H(t)\neq 1$, its degree $n$
is positive and its degree $n$ component is $c.e_\delta$ for some
$c\neq 0$. Let $\delta'\in\P^1_K$ be a line distinct 
from $\delta$
and let $v\in\Omega_{\delta'}$. Since $e_\delta.u\neq 0$ for
any non-zero $u\in\delta'$, we have

\centerline{$\hc(H(t).v)=ce_\delta \hc(v)$,}

\noindent for any $v\in \Omega_{\delta'}$. It follows that
$E^*_\delta.\Omega_{\delta'}\subset\Omega_\delta$ for any
$\delta\neq\delta'$.

Therefore by Lemma \ref{Ping-Pong}, the free product
$*_{\delta\in\P^1_K}\,E_\delta$ embeds in $\GL_1(2,K[t])$.
Hence by Lemma \ref{generation}, we have
$*_{\delta\in\P^1_K}\,E_\delta=\GL_1(2,K[t])$.

\end{proof}

\bigskip

\noindent {\it Remark.}  
Set

$U^+=\{\begin{pmatrix} 1& f\\ 0&1 \end{pmatrix}
\vert\,\textnormal{ for}\, f\in K[t]\},$

$U^-=\{\begin{pmatrix} 1& 0\\ f&1 \end{pmatrix}
\vert\,\textnormal{ for}\, f\in tK[t]\},$ and

$U=\{\begin{pmatrix} 1& x\\ 0&1 \end{pmatrix}
\vert\,\textnormal{ for}\, x\in K\}$.

It is easy to show that Lemma \ref{Tits} is equivalent to
the fact that $U^+*U^-=\GL_U(2,K)$. This results in stated in the context of Kac-Moody groups in Tits notes \cite{T82}.
Since these notes are not widely distributed, let mention that an equivalent result is stated in \cite{T89}, Section 3.2 and 3.2. For Tits original proofs, see \cite{T87}.

\section{The Linear Representation of  
$\Aut_{1}\,K^2$}

In this section, 
we prove Theorem B.

\bigskip\noindent
{\it 7.1 The subgroups $F_\delta$ in $\Aut_{1}\,K^2$.}

\noindent
Let $\delta\in\P^1_K$ and let $(a,b)\in\delta$ be nonzero.
For any $f\in  K[t]$, let $\tau_\delta(f)$ be the 
automorphism

\centerline {$\tau_\delta(f): 
(x,y)\to (x+af(bx-ay), y+bf(bx-ay))$.} 

\noindent We have 
$\tau_\delta(f)\circ \tau_\delta(g)=\tau_\delta(f+g)$. Let
$F_\delta =\{\delta(f)\vert \,f\in t^2 K[t]\}$. Indeed for
$\delta_0=K(0,1)$, $\tau_{\delta_0}(f)$ is the elementary
automorphism

\centerline{$(x,y)\mapsto (x,y+f(x))$,}

\noindent and  $F_{\delta_0}$ is the group 
$\Elem_1(K):=\Elem(K)\cap\Aut_1\,K^2$.  In general, we have
$F_\delta=\Elem_1(K)^g$, for any $g\in \GL(K^2)$ such that
$g.\delta_0=\delta$.

\begin{lemma}\label{freeAut}  We have

\centerline{$\Aut_1(K^2)= *_{\delta\in\P^1_K}\,F_\delta$.}

\end{lemma}

\begin{proof} 
First we check that $\Aut_0\,K^2$ satisfies the hypotheses of Lemma \ref{mixing1}. By van der Kulk Theorem and Lemma \ref{subamal}, we have 

\centerline{$\Aut_0\,K^2=G_1*_A\,G_2=\Gamma_0$}

\noindent where $G_1=\GL(2,K)$, $G_2=\Elem_0(K)$ and
$A=B_0(K)$ is the Borel subgroup of $\GL(2,K)$. We have
$G_2=A\ltimes E$ where $E=\Elem_1(K^2)$. Clearly the
map $\chi:G_1*_A\,G_2\to G_1$ is simply the map

\centerline{$\phi\in\Aut_0(K^2)\to \textnormal{d}\phi_{\bf 0}$,}

\noindent and its kernel is $\Gamma_1=\Aut_1(K^2)$.
Since $G_1/A=\GL(2,K)/B_0(K)=\P^1_K$, and 
$E^\gamma=F_\delta$ for any $\gamma\in \GL(2,K)$ with
$\gamma.\delta_0=\delta$,
it follows from
Lemma \ref{mixing1} that

\centerline{$\Aut_1(K^2)=*_{\gamma\in\P^1_K}\,F_\delta$.}

\end{proof}

\bigskip
\noindent{\it 7.2 The isomorphism $\Psi$}

\noindent For any $\delta\in\P^1_K$, let 
$\psi_\delta:F_\delta\to E_\delta$ be the isomorphism
defined  by

\centerline{$\psi_\delta(\tau_\delta(f))=\id +f/t\otimes e_\delta$.}

\begin{lemma}\label{Psi} The collection of isomorphisms 
$(\psi_\delta)_{\delta_{\P^1_K}}$ induces an isomorphism

\centerline{$\psi:\Aut_1\,K^2\simeq \GL_1(2,K[t])$.}

\end{lemma}

\begin{proof} the statement is a consequence of
Lemmas \ref{Tits} and \ref{freeAut}.
\end{proof}

\bigskip
\noindent{\it Remark} The formula defining $\psi$ is complicated.
Indeed if an automorphism $\sigma\in \Aut_1(K^2)$ is written
as a free product $\sigma=\sigma_1\dots\sigma_m$, where $\sigma_i\in F_{\delta_i}$ has degree $m_i$, then $\sigma$ has degree
$m_1\dots m_n$ but $\psi(\sigma)$ has degree
$m_1+\dots +m_m-m$.

\bigskip
\noindent {\it 7.3 Proof of Theorem B}

\begin{lemma}\label{Card}
Let $K$, $L$ be  fields. If
$\Card\, K=\Card\, L$ and $\ch\, K=\ch\, L$,
then $\Aut_1\,K^2$ and $\Aut_1\,L^2$
are isomorphic.

\end{lemma}

\noindent By definition the {\it prime fields} are the fields $\Q$ and $\F_p$.

\begin{proof} It can be assumed that $K$ is infinite.
Let $F$ be its prime field, let $P$ be a set and let
$E$ be a $F$-vector space with

\centerline{$\Card P=\Card K$ and $\dim_F\,E=\Card\,K$.}

By lemma \ref{freeAut}, we have

\centerline{$\Aut_1\,K^2\simeq *_{p\in P}\, E_p$,}

\noindent where each $E_p$ is a copy of $E$. It follows that 
 $\Aut_1\,K^2$ and $\Aut_1\,L^2$
are isomorphic if $\Card\, K=\Card\, L$ and $\ch\, K=\ch\, L$.
\end{proof}

\begin{MainB}
For any  field $K$, the group $\Aut_1\,K^2$ is
 linear over $K(t)$. 
 
 Moreover if
 $K\supset k(t)$ for some infinite field $k$, then
 there exists an embedding 
 $\Aut_1\,K^2\subset SL(2,K)$.
\end{MainB}

\begin{proof} It follows from Lemma \ref{Psi} that
$\Aut_1\,K^2\subset \SL(2,K(t)$, and therefore 
$\Aut_1\,K^2$ is linear over $K(t)$.

Assume now that  $K\supset k(t)$ for some infinite field $k$. We claim that there exists a field $L$ with
$L(t)\subset K$ and $\Card\,L=\Card\,K$. If
$\Card\, K=\aleph_0$, then the subfield $k$ satisfies the claim. Otherwise, we have $\deg\,K> \aleph_0$
and there is an embedding $L(t)\subset K$ for some
subfield $L$ with $\deg\,L=\deg\,K$. Since
$\deg\,L=\Card L=\Card K$, the claim is proved.

It follows from Lemma \ref{Card} and \ref{Psi} that 

\centerline{$\Aut_1\,K^2\simeq \Aut_1\,L^2\subset \SL(2,L(t))
\subset\SL(2,K)$. }

 \end{proof}

\section{Linearity of $\Aut_S\,K^2$ for some $S\subset \GL(2,K)$}

Let $K$ be a field and let $F$ be its prime subfield.

\bigskip\noindent
{\it 8.1 Abelian groups} 

\noindent Let $A$ be a torsion free abelian group, where the product is denoted multiplicatively. Let $F(A)$ be the fraction field of the group algebra $F[A]$. Viewed as a $F[A]$-module,
$F(A)$ is called the  {\it standard $F[A]$-module}, and
it is denoted $\St(A)$. 
A direct sum of standard module $M$ is indeed a 
$F(A)$-vector space, and the multiplicity of
$\St(A)$ dans $M$, denoted by  
$[M:\St(A)]$, is $\dim_{F(A)}\,M$.

For any $n\geq 1$, let
$_n\rho:A\to A$ be the group morphism  $a\mapsto a^n$.
For a $F[A]$-module $M$, the $F[A]$-module 
$_n\rho_*\,M$ is denoted $M^{(n)}$.

\begin{lemma}\label{standard} The $F[A]$-module $F(A)^{(n)}$ is a 
direct sum of standard modules, and
$[F(A)^{(n)}:\St(A)]=[A:A^n]$.
\end{lemma}
 
\begin{proof} Since $F[A]$ is an algebraic extension of $F[A^n]$, it follows that
$F(A)=F(A^n)\otimes_{F[A^n]}\,F[A]$, and therefore

\centerline{$F(A)=\oplus_{a\in T}\,  F(A^n)a$}

\noindent where $T$ is a set of representative of 
$A/A^n$, from which the Lemma follows.

\end{proof}

\bigskip\noindent
{\it 8.2 The condition ${\cal G}$}

\noindent
Let $A$ be a subgroup of $K^*$.
By definition, its {\it rank}  is the cardinal $\rk\,A:=\dim_{\Q}\,\Q\otimes A$ and its {\it degree}, denoted as $\deg\,A$, is the transcendental degree of the subfield generated by $A$. It is clear that  
$\deg\,A\leq \rk\,A$. Moreover, 
we have  $\deg\,A'=\rk\,A'$ for any finite rank
subgroup $A'\subset A$ iff the map
$F[A]\to K$ is one-to-one, and therefore
$K$ contains the field $F(A)$.

Let $S$ be a subgroup
of $\SL(2,K)$. For any $\delta\in\P^1_K$,
set  $S_\delta:=\{g\in S\vert g.\delta=\delta\}$.
By definition, any  $g\in S_\delta$ acts over $\delta$ by multiplication by some scalar  $\chi(g)$ and
the map 
$\chi:g\in S_\delta \mapsto \chi(g)\in K^*$
is a group morphism. Set
$A(\delta)=\chi(S_\delta)$.

By definition, we say that $S$ {\it has the property
${\cal G}$  in $\SL(2,K)$}  if for any $\delta\in\P^1_K$, 
the group $A(\delta)$ is torsion free,
and  we have $\deg\,A=\rk\,A$ for any 
finite rank subgroup $A$ in $A(\delta)$. 
As usual, $S$ is called {\it virtually ${\cal G}$
in $\SL(2,K)$}  if some finite index subgroup $S'\subset S$ has the property ${\cal G}$.

\bigskip\noindent
{\it 8.3 The Theorem C}

Let $S$ be a subgroup of $\SL(2,K)$. Recal that

\centerline{$\GL_S(2,K[t])=\{G\in \GL(2,K[t])\vert G(1)\in S\}$.}

\noindent note that we have indeed
$\GL_S(2,K[t])\subset \SL(2,K[t])$. 

\begin{MainC}\label{mainC} Assume that the subgroup $S$ has the property ${\cal G}$ in $\SL(2,K)$. Then $\Aut_S\,K^2$ is linear
over some field extension of $K$. 

Moreover if $\deg\,K\leq \aleph_0$, then we have

\centerline{$\Aut_S\,K^2\simeq \GL_S(2,K[t])$.}

In particular, $\Aut_S\,K^2$ is linear over $K(t)$.
\end{MainC}

\begin{proof} 
There exists  a field extension $L\supset K$,
which satisfies one of the following two  hypotheses

$({\cal I}_1)$\hskip 1cm $[L:K]\geq \deg K$ if
$\deg\,K>\aleph_0$, or

$({\cal I}_2)$\hskip 1cm $L=K$  if $\deg K\leq \aleph_0$

It follows from Lemmas \ref{Tits} and \ref{freeAut}
that 

\centerline{
$\Aut_1\,K^2=*_{\delta\in\P^1_K}\,F_\delta$.}

\noindent Using that $\P^1_K\subset \P^1_L$, we also have

\centerline{$\GL_1(2,L[t])\supset *_{\delta\in\P^1_K}\,E_\delta$.}

\smallskip\noindent
{\it Step 1: existence of some 
$S_\delta$-equivariant embeddings
$F_\delta\to E_\delta$.} Let $\delta\in \P^1_K$. We claim that 

$({\cal I}_1)$ implies the existence of a 
$S_\delta$-equivariant embedding
$\psi_\delta:F_\delta\to E_\delta$,

$({\cal I}_2)$ implies the existence
a $S_\delta$-equivariant isomorphism
$\psi_\delta:F_\delta\to E_\delta$.

\noindent Note that 
the kernel of the morphism 
$\chi: S_\delta\to A(\delta)$  acts trivially
on $F_\delta$ and $E_\delta$, so both of them are
$F[A(\delta)]$-modules, where $F$ is the prime field of $K$.

Note that  $e_\delta^g=\chi(g)^2 e_\delta$ for any 
$g\in S_\delta$. Hence 
 $E_\delta\simeq \oplus_{n\geq 1} L.z^ne_\delta$ is a countable sum of modules isomorphic to $L^{(2)}$. It follows from Lemma \ref{standard} that $E_\delta$ is a direct sum of standard modules with multiplicity

\centerline{
$[E_\delta:\St(A(\delta)]
=\aleph_0 [L:K][K:F(A\delta))]
[A(\delta):A^2(\delta)]$.}

Let $\phi_n$ be the automorphism
$(x,y)\mapsto (x,y+x^n)$. Note that
$\phi_n^g$ is the automorphism
$(x,y)\mapsto (x,y+ \chi(g)^{n+1} x^n$ for any 
$g\in S_\delta$. Hence we have
 $F_\delta\simeq \oplus_{n\geq 3}\, K^{(n)}$.
  
  By Lemma \ref{standard}, $F_\delta$ is a direct sum of standard modules with multiplicity

\centerline{
$[F_\delta:\St(A(\delta)]
=\sum_{n\geq 3}\, [K:F(A\delta)]
[A(\delta):A^n(\delta)]$.}

Note that $[A(\delta):A^n(\delta)]$ is obviously
$\leq\deg\,K$. It follows from hypothesis
$({\cal I}_1)$ that 
$[F_\delta:\St(A(\delta)]\leq [E_\delta:\St(A(\delta)]$
and from hypothesis $({\cal I}_1)$ that 
$[F_\delta:\St(A(\delta)]= [E_\delta:\St(A(\delta)]$.
Therefore both claims are proved.

\smallskip\noindent
{\it Step 2: end of the proof.}
Let $\Omega\subset \P^1_K$ be a set of
representative of $\P^1_K/S$. 
By Lemma \ref{mixing3}, the collection of morphisms
$(\psi_\delta)_{\delta\in\Omega}$ provides

(i) an embedding $\Aut_S\,K^2\subset \GL_S(2,L[t])$ under hypothesis $({\cal I}_1)$,

(ii) an isomorphism 
$\Aut_S\,K^2\simeq \GL_S(2,K[t])$ under hypothesis 
$({\cal I}_2)$,

\noindent from which the lemma follows.

\end{proof}

\section{Some Example of Linear subgroups of 
$\Aut\,K^2$}

Some corollaries of Theorem  B are stated in Section 9.1. Next, we give example of 
groups which are virtually $({\cal G})$. First,
in subsection 9.3, 
there are examples with $A(\tau)=\{\pm 1\}$. Next
there are examples for which some $A(\tau)$ are arbitrarily large. The last subsection 9.5 show a more
complicated examples where different groups $A(\tau)$ are disctinct, and their rank is not constant.

\bigskip
\noindent
{\it 9.1 Linearity of $\Aut\,K^2$ for a finite field $K$}

\begin{CorD1} Let $K$ be a finite field of characteristic 
$p$. The
group $\Aut\,K^2$ is linear over $\F_p(t)$.

\end{CorD1}
 
 \begin{proof} By Theorem B, the group
 $\Aut_1\,K^2$ is linear over $K(t)$. Since the index
  $[\Aut\,K^2:\Aut_1\,K^2]$ is finite, the group
 $\Aut\,K^2$ is also  linear over $\F_p(t)$.
 
\end{proof}

\begin{CorD2} Let $K$ be an infinite subfield of
$\overline F_p$.

The group $\Aut\,K^2$ is not linear, even over a ring.
However any FG subgroup is linear over
$F_p(t)$.

\end{CorD2}

\begin{proof} Indeed $\Aut\,K^2$ is not linear by
Theorem B. Any FG subgroup $\Gamma$ of $\Aut\,K^2$  is a subgroup
of $\Aut\,L^{2}$ for some finite subfield $L\subset K$, so it is linear by Corollary D.1.
\end{proof}

\bigskip\noindent

 \bigskip\noindent
{\it 9.2 Examples of groups $S$ with trivial $A(\delta)$}

\noindent {\it Example 1.} Set $U=\{$
$\begin{pmatrix}
1 & x \\
0 & 1
\end{pmatrix}$ $\vert x\in K\}$. Obviously,
$U$ satisfy condition $({\cal G})$ and therefore
$\Aut_U\,K^2\simeq GL_U(2,K[z])$ is linear
over $K(z)$.
Consequently $\Aut\,K^2$ contains a linear subgroup
of codimension 5.

\smallskip
\noindent {\it Example 2.} Let $L\subset K$ be a quadratic extension. Set 
$S=\{z\in L^*\vert N_{L/K}(z)=1\}$, and let
$S^{\infty}$ be the subgroup of all elements 
of $S$ or order a power of $2$. The group
$S$ can be viewed a subgroup of $\SL(2,K)$.
If $S^{\infty}$ is finite, then 
$\Aut_S\,K^2$ is linear. 

Indeed by Baer's theorem \cite{Ba}, we have 
$S=S^\infty\times S'$ for some subgroup $S'\subset S$.
We have $S'_\delta=\{1\}$ for any $\delta\in\P^1_K$ and
therefore $S'$ satisfies the property 
${\cal G}$. Therefore $S$ is virtually ${\cal G}$ in
$SL(2,K)$. Thus follows from Theorem C that 
$\Aut_S\,K^2$ is linear over a field extension of $K$.

  \bigskip\noindent
{\it 9.3 Examples of groups $S$ with  $A(\delta)$
independent of $\delta$}

\noindent
{\it Example 3: the subgroup $\SL(2,\Z[t,t^{-1}])$ of $\SL(2,\Q(t))$}
For this group, we have
$A(\delta)=\{\pm 1\}\times \{t^n\vert n\in \Z\}$
for any $\delta \in \P^1_{\Q(t)}$
To get rid of the torsion part of $A(\delta)$, it is enough to consider the finite index subgroup

\centerline{ $\Gamma=\{G\in 
 \SL(2,\Z[t,t^{-1}])\vert \,G(1)=\id \mod 3\}$.}
 
  The subgroup $\Gamma$ has the property $({\cal G})$,
 therefore by Theorem C the group  
 $\Aut_{\SL(2,\Z[t,t^{-1}])}\,\Q(t)^2$ is linear over a field of characteristic zero.
 
 More generally, let $L$ be a quadratic complex 
 number field and let ${\cal O}$ be its ring of integers. Since the group of units in ${\cal O}$
 is finite, the group $\SL(2,{\cal O}[t,t^{-1}])$
 is also virtually ${\cal G}$. Hence the group
 $\Aut_{\SL(2,{\cal O}[t,t^{-1}])}\,L(t)^2)$  is linear
 over a field of characteristic zero..
 
Similarly the group $\Aut_{\SL(2,\F_q[t,t^{-1}])}\,\F_q(t)^2$ is linear over some  field extension of $\F_q$.
 
 \smallskip
 \noindent
{\it Example 4: groups with arbitrarily large $A(\delta)$}
 Let $A$ be a torsion free additive group of arbitrary rank. Let $S=\SL(2,\Z[A])$.
 We have 
 $A(\delta)= \{\pm 1\}\times S$ for any $\delta \in \P^1_{\Q(M)}$. As before, there is a finite
 index subgroup $\Gamma$ that does not contains
 quasi-unipotent elements of quasi-order 2.
 Therefore $\Aut_{\SL(2,\Z[A])}\,\Q(A)^2$ is linear
 over a field of characteristic zero.

  \bigskip\noindent
{\it 9.4 Examples of groups $S$ with  $A(\delta)$
depending on $\delta$}

Last, we shall provide more interesting examples, where
the hypothesis ${\cal G}$ is fully used.

\begin{lemma}\label{rank1} Let $R$ be a prime
normal ring with fraction field $K$  and let
$L$ be a quadratic extension of $K(t_1,\dots,t_m)$
for some integer $m\geq  1$. Let 
$B$ be the integral closure of $R[t_1,\dots,t_m]$ in 
$L$. Assume moreover that ${\overline K}\cap L=K$.
Then  $B^*/R^*$ is isomorphic to $\{1\}$ or $\Z$.
\end{lemma}

\begin{proof} 
Recall that $\Spec\,K[t_1,\dots t_m]$
is the affine space $\na_m$. Set 
$C=\Spec \,K\otimes B$. There is a unique normal
compactification $\overline C$ of $C$ such that
the finite map
$C\to\na_m$ extends to a finite map
$\pi:\overline C\to \P_m$. We will 
identify $\P_{m-1}$ with $\P_m\setminus\na_m$.

For a divisor $Z$ in
$\overline C$, let $v_Z$ be the corresponding valuation
of the field $L$.
Since $[L:K(t_1,\dots,t_m)]=2$, 
$\pi^{-1}(\P_{m-1})$ contains at most two
irreducible divisors.

Assume that $Z:=\pi^{-1}(\P_{m-1})$  is irreducible.
For any $f\in B^*$, we have $v_Z(f)\geq 0$ or 
$v_Z(f^{-1})\geq 0$, hence $f$ or $f^{-1}$ is
defined on  $\overline C$. It follows that $f$ is constant, hence it follows that
$B^*=K^*$.

Assume now that $\pi^{-1}(\P_{m-1})=Z_1\cup Z_2$,
where $Z_1$, $Z_2$ are two distinct divisors.
Let $v:B^*\to\Z^2$ be the map 
$f\in B^*\mapsto (v_{Z_1}(f), v_{Z_2}(f))$. 
Since functions defined on  $\overline C$ are constant, that there are no $f\in B^*$ with
$v_{Z_1}(f)>0$ and   $v_{Z_2}(f)>0$. 
Hence $\Im\,v$ has rank $\leq 1$. Since $\Ker\,v=K^*$, it follows that 
$(K\otimes B^*)/K^*$ is isomorphic to $\{1\}$ or $\Z$.

Since $R$ is normal, we have $B\cap K=R$, hence 
$B^*/R^*\simeq (K\otimes B^*)/K^*$  is isomorphic to $\{1\}$ or $\Z$.

\end{proof}

\smallskip
 \noindent
{\it Example 5.} Let $A$ be a torsion free additive group of arbitrary rank.  Let 
$S=\SL(2,\Z[A][t_1,\dots,t_m])$ and 
$L=\Q(A)((t_1,\dots,t_m))$. 
We will compute the groups $A(\delta)$ for
$\tau\in\P^1_{L}$. 

First, if $\tau$ belongs to $\P^1_{K}$, where 
$K=\Q(A)(t)$,  
we have $A(\delta)= \{\pm 1\}\times A$ as before.
Next assume that $\tau\notin \P^1_{K}$ is defined over
a quadratic extension $K'\subset L$ of $K$.
Let $B$ be the integral closure of
$\Z[A][t_1,\dots,t_m]$ in $K'$. Then we have 
$A(\tau)=\{\zeta\in B^*\vert N_{K'/K}(\zeta)=1\}$,
where $N_{K'/K}$ is the norm map.
It follows from Lemma \ref{rank1}   that
$A(\tau)=\{\pm 1\}$ or $A(\tau)=\{\pm 1\}\times\Z$.
Last, if $\delta$ is not defined over a quadratic extension of $K$, then we have $A(\delta)=\{\pm 1\}$.

Define the morphism 
$F\in \Z[A][t_1,\dots,t_m]\mapsto F(1,0)\in\Z$ by
the requirements $e^a\mapsto 1$ for $a\in A$ and
$h_i\mapsto 0$. As before, set 

\centerline{ $\Gamma=\{G\in 
 \SL(2,\Z[A][h_1,\dots,h_m])\vert \,G(1,0)=\id \mod 3\}$.}
 
It follows that the finite index subgroup
$\Gamma$  has the property  ${\cal G}$ in $SL(2,L)$. 
Therefore the group
$\Aut_{S}\,L^2$ is linear over a field of characteristic zero.

\section{Nonlinearity of finite codimension automorphisms groups of $K^n$, for $n\geq 3$}

The previous Theorem B shows that the group of polynomial automorphisms of $K^2$ contains some
finite codimension linear subgroups. However, this result does not extend to $K^n$, for 
$n\geq 3$, as it will be shown in this section. 
 
 For our purpose, the case $n=3$ is enough. 
In what follows, the coordinates of $K^3$ will be written as $(z,x,y)$. Let $\Aut\,K^3$ be the group of polynomial automorphisms of $K^3$ and let
$\TAut\,K^3\subset \Aut\,K^3$ be the the subgroup of tame automorphisms.  Since both $\TAut\,K^3$ and 
$\Aut\,K^3$  are higher dimensional analogs of 
$\Aut\,K^2$, we will investigate the nonlinearity of subgroups in  $\TAut\,K^3$ as well.
 
 Let ${\bf m}$ be a
 finite codimensional ideal in $K[z,x,y]$. Let
 $\Aut_{\bf m}\,K^3$ be the group of all polynomial automorphisms $\phi$ of the form
 
 \centerline{$(z,x,y)\mapsto 
 (z+f, x+g, y+h)$,}
 
 \noindent where $f,\,h$ and $g$ belongs to ${\bf m}$.
 This condition means that $\phi$ fixes some infinitesimal neighborhood of a finite collection of points in $K^3$.
Set $\TAut_{\bf m}\,K^3=
\TAut\,K^3 \cap\Aut_{\bf m}\,K^3$.
 
 In this section we are going to prove that 
 
 \begin{MainD} For any finite codimensional ideal ${\bf m}$
 of $K[z,x,y]$, the group $\Aut_{\bf m}\,K^3$ and
 $\TAut_{\bf m}\,K^3$ are not linear, even over a ring.
  \end{MainD}
 
 It is plausible that any 
 finite codimensional subgroup of $\Aut\,K^3$ or
 $\TAut\,K^3$ contains 
 $\Aut_{\bf m}\,K^3$ or $\TAut_{\bf m}\,K^3$  for 
 some ${\bf m}$. Unformally speaking, Theorem D would mean that these groups do not contain linear finite codimension subgroups.

\bigskip\noindent
{\it 10.1 Nilpotency index of some p-groups}

\noindent 
The {\it nilpotency index} of a nilpotent group is
the lenght of its ascending central series.

Let $p$ be a prime integer, and let $E$ be an
elementary $p$-group of rank $r$, i.e. $E$ is a
$\F_p$-vector space of dimension $r$. Note that  $E$ acts by translation on the group algebra $\F_p[E]$ and set
$G(r)=E\ltimes \F_p[E]$. Obviously the group $G(1)$ is nilpotent of nilpotency index  $p$. Since
$G(r)\simeq G(1)^r$, it follows that

\begin{lemma}\label{G(r)}
The nilpotency index of the nilpotent group $G(r)$ is $pr$.
\end{lemma}

Let $M$ be a cyclic $\F_p[E]$-module generated by some $f\in M$. 

\begin{lemma}\label{free} If $\sum_{u\in E} \, u.f\neq 0$, then  the $\F_p[E]$-module $M$ is free of rank one.
\end{lemma}

\begin{proof}Set $N=\sum_{u\in E}\,e^u$,
where $(e^u)_{u\in E}$ is the usual basis of
$\F_p[E]$. 
Note that $\F_p.N$ is the space of $E$-invariant in
$\F_p[E]$. Since any nonzero $\F_p[E]$-module contains some
nonzero $E$-invariant vector, any nonzero ideal of
$\F_p[E]$ contains $N$. It follows that the annihilator of $f$ is zero, and therefore $M$ is free of rank one.
\end{proof}

\bigskip\noindent
{\it 10.2 The nil group $G(I)$}

\noindent
Let $F$ be a prime field and let $A$ be a $F$-commutative algebra. 
Since $A$ acts by translation
on $A[t]$,  one can consider the semi-direct product
$G(A):=A\ltimes A[t]$. For an ideal
$I$ of $A$, let $G(I)$ be the kernel of the map
$G(A)\to G(A/I)$. 

Let $E\subset A$ be an nonzero additive subgroup, let 
 $f(t)\in A[t]\setminus 0$ and let $M$ be the additive subgoup generated 
 by all polynomials $f(t+u)$ when $u$ runs over $E$.
 The group $E\ltimes M$, which  is a subgroup of $G(A)$,
 is obvioulsly nilpotent.

\begin{lemma}\label{index}
Assume that the algebra $A$ is prime.

(i) If $F=\Q$,  the nilpotency index of $E\ltimes M$ is  $\deg\,f$.

(ii) Assume that $F=\F_p$, that $E$ is an elementary
group of rank $r$ and that $f(t)=a x^{p^r-1}$ for some
$a\in A\setminus 0$. Then the group $E\ltimes M$ has nilpotency index 
$rp$.
\end{lemma}

\begin{proof} The first statement is obvious, and the proof
will be skipped. From now on, we will assume that $F=\F_p$.

\noindent {\it Step 1.} We claim that 

\centerline{$\sum_{u\in E} \,u^{p^r-1}=c
\prod_{u\in E\setminus 0}\, u$.}

\noindent for some $c\in \F_p$ independent of $A$ and 
$E$.
It is enough to prove the claim for
$A=\F_p[x_1,\dots,x_r]$ and $E=\oplus_{i} \F_p.x_i$.

Let $u\in E\setminus 0$ and let $H\subset E$ be an 
hyperplane with $u\notin H$. For any integer $n\geq 0$, we have

\centerline{$\sum_{\lambda\in \F_p}\,\lambda^n=0$}

\noindent except if $n$ is a positive multiple of $p-1$. Therefore we have

$\sum_{\lambda\in \F_p}\,(\lambda u+ v)^{p^r-1}
=\sum_{n>0} \,c_n u^{n(p-1)} v^{p^r-1-n(p-1)}$

\noindent for some $c_n\in \F_p$, for any $v\in H$. 
Since any $w\in E$ can be uniquely written as
$w=\lambda u+v$, for some $\lambda\in\F_p$ and 
$v\in H$, it follows that the polynomial
$\sum_{w\in E} \,w^{p^r-1}$ is divisible by $u^{p-1}$ for
any $u\in E\setminus 0$. Therefore the polynomial
$\sum_{w\in E} \,w^{p^r-1}$ is divisible by 
$\prod_{u\in E\setminus 0}\, u$. Since both polynomials have
degree $p^r-1$, it follows that
$\sum_{u\in E} \,u^{p^r-1}=c
\prod_{u\in E\setminus 0}\, u$, for some $c\in F_p$.

\noindent {\it Step 2.}
 Next we prove that $c=1$. Since $c$ is a universal constant, it can be computed for $A=E=\F_{p^r}$.
Since 

\centerline{
$\sum_{\lambda\in \F_p^r}\,\lambda^{p^r-1}=-1$, and
$\prod_{\lambda\in \F_p^r}\,\lambda=-1$,}

\noindent it follows that $c=1$.

\noindent {\it Step 3.} Set 
$g(t)=\sum_{u\in E}\,f(t+u)$. We have 
$g(0)=\sum_{u\in E} \,a u^{p^r-1}$, therefore
$g(0)=a\prod_{u\in E\setminus 0}\, u$. Since
$g\neq 0$, it follows from Lemma \ref{free} that
the $\F_p[E]$-module $M$ is free of rank one.
Therefore $E\ltimes M$ is isomorphic to $G(r)$. Thus
its nilpotency index is $pr$ by 
by Lemma \ref{G(r)}.

\end{proof}
A obvious corollary of the previous lemma is

\begin{lemma}\label{G(I)} Assume that the algebra $A$ is prime and let $I$ be a nonzero ideal. Then $G(I)$ contains subgroups of arbitrarily large nilpotency index.
\end{lemma}

\begin{proof}
The lemma follows directly from Lemma
\ref{index}. When $F=\F_p$, on should add the remark that the hypotheses imply that $\dim_{\F_p}\,I$ is necessarily infinite.
\end{proof}

\bigskip\noindent
{\it 10.3 The group $\Elem(I)$}

\noindent From now on, let $I$ be a nonzero ideal in
$K[z]$. Let $\Elem(I)$ be the kernel of 
the map $\Elem(K[z])\to \Elem(K[z]/I)$.

\begin{lemma}\label{ElemI}
The group $\Elem(I)$ is nil, but it contains subgroups
of abitrarily large nilpotency index.
\end{lemma}

\begin{proof} Note that $1$ is the unique 
inversible element in $1+I$. 
Therefore 
$\Elem(I)$  consists of automorphisms
$\phi$ of the form

\centerline{$\phi:(x,y)\mapsto (x+u, y+f(x))$,}

\noindent for some $u\in I$ and $f\in I[x]$,
where  $I[x]$ denotes 
the set of all polynomials $f(x)\in K[x]$ whose coefficients are in $I$. It follows that
$\Elem(I)$ is isomorphic to the group 
$G(I)$, and therefore $\Elem(I)$ is nil, but contains subgroups
of abitrarily large nilpotency index by Lemma
\ref{G(I)}.

\end{proof}

\bigskip\noindent
{\it 10.4 The amalgamated product 
$\Aff(2,I)*_{B(I)}\Elem(I)$}

\noindent
Let $\Aff(2,I)$ and $B(I)$ be the
kernel of the maps
$\Aff(2,K[z])\to \Aff(2,K[z]/I)$ and
$B(K[z])\to B(K[z]/I)$.

\begin{lemma}\label{HI} The group
$\Aff(2,I)*_{B(I)}\Elem(I)$ is not
linear, even over a ring.
\end{lemma}

\begin{proof} By Lemma \ref{ElemI}, the group
$\Elem(I)$ is not linear over a field. Therefore
by Lemma \ref{criterion}, it is enough to show that
the amalgamated product $\Gamma:=\Aff(2,I)*_{B(I)}\Elem(I)$ satisfies
the hypothesis ${\cal H}$.

In order to do so, we first define two 
specific automorphisms $\gamma$ and $\phi$ as follows.
Let $r\in I\setminus 0$ and
let $n\geq 3$ be an integer prime to 
$\ch\,K$. Let $\gamma\in \Aff(2,I)$ be the
linear map $(x,y)\mapsto (x+ry,y)$ and let
$\phi\in \Elem(I)$ be the polynomial
automorphisms 
$(x,y)\mapsto (x,y+rx^n)$.

Let $g$ be an arbitrary element of  $B(I)\setminus 1$.
By definition,
$g$ is an affine map
$(x,y)\mapsto(x+u,y+v+wx)$ for some $u,\,v.\,w\in I$.

If $w\neq 0$,  the linear part of
$g^\gamma$ is not upper triangular, therefore
$g^\gamma$ is not in $B(I)$.  If $w=0$ but $u\neq 0$, then $g^\phi$ is an automorphism of degree
$n-1$ and its leading term is
$(x,y)\mapsto (0, nru x^{n-1})$, therefore 
$g^\phi$ is not in $B(I)$. Last if $u=w=0$,
then $v$ is not equal to zero and
$g^\gamma$ is the affine transformation
$(x,y)\mapsto(x+vr,y+v)$. Since $vr\neq 0$, it follows
that $g^{\gamma\phi}$  is an automorphism of degree
$n-1$ and its leading term is
$(x,y)\mapsto (0, nr^2v x^{n-1})$, therefore 
$g^{\gamma\phi}$ is not in $B(I)$.

It follows that, for any $g\in B(I)\setminus 1$
at least one of the three elements
$g^\gamma$, $g^\phi$ or $g^{\gamma\phi}$ is not in $B(KI)$. Therefore hypothesis ${\cal H}$ holds.

\end{proof}

\bigskip\noindent
{\it 10.5 The group $\TAut\,K^3$ of tame automorphisms}

\noindent An
automorphism of $K^3$ of the form

\centerline{$(z,x,y)\mapsto (z, x+f(z), y+g(z,x))$}

\noindent where $f(z)\in K[z]$ and $g(z,x)\in K[z,x]$
is called {\it triangular}. Let $T(3,K)$ be the group
of all triangular automorphisms, and let
$\Aff(3,K)$ be the group of affine automorphisms of $K^3$. By definition, the group $\TAut\,K^3$ of {\it
tame automorphisms} of $K^3$ is the subgroup of
$\Aut\,K^3$ generated by 
$T(3,K)$ and $\Aff(3,K)$. It has been proved by Shestakov and Urmibaev \cite{SU} that 
$\TAut\,K^3\neq \Aut\,K^3$.

Note that $\Aut\,K[z]^2=\Aut_{K[z]}\, K[z,x,y]$ is 
obvioulsly a subgroup of 
$\Aut\,K^3=\Aut_K\, K[z,x,y]$.

\begin{lemma}\label{tame} Let $I$ be a nonzro ideal of
$K[z]$. The groups $\Elem(I)$
and $\Aff(2,I)$ are subgroups of
$\TAut\,K^3$.
\end{lemma}

\begin{proof} Any automorphism in $\Elem(I)$
is of the form

\centerline{$(z,x,y)\mapsto (z, x+f(z), y+g(z,x))$}

\noindent with $f(z)\in I$ and $g(x,y)\in I[y]$, therefore $\Elem(K[z],I)$ is a subgroup of $T(3,K)$.

We claim that
$\Aff(2,K[z])\subset \TAut\,K^3$.
We have 

\centerline{$\Aff(2,K[z])=T.\GL(2,K).\GL_1(2,K[z])$,}

\noindent 
where $T$ is the group of translations of the form
$(z,x,y)\mapsto (z,x+f(z),y+g(z))$. Since
$T\subset T(3)$ and  $\GL(2,K)\subset \Aff(3,K)$,
it remains to prove that any 
automorphism in $\GL_1(2,K[z])$ is tame.

Let $E$ be the group of automorphisms of the form

\centerline{$(z,x,y)\mapsto (z, x, y+f(z)x)$}

\noindent with $f(0)=0$. It follows from 
Lemma \ref{generation} that $\GL_1(2,K[z])$
is generated by all subgroups $E^g$, when $g$ runs over
$\GL(2,K)$. Since $E$ is a group of triangular automorphism, it follows that 
$\GL_1(2,K)\subset \TAut\,K^3$.

Therefore, we have 
$\Aff(2,I)\subset \Aff(2,K[z])\subset \TAut\,K^3$.

\end{proof}

\bigskip\noindent
{\it 10.6 Proof of Theorem D} 

\begin{proof}
Set $I={\bf m}\cap K[z]$. Since 
${\bf m}$ has finite codimension in $K[z,x,y]$, the ideal $I$ is nonzero.

Set $\Gamma=\Aff(2,I)*_{B(I)}\Elem(I)$.
We have
$\Aff(2,I)\cap\Elem(I)=B(I)$.
Thus by Lemma \ref{subamal}, $\Gamma$ is
a subgroup of $\Aff(2,K(z))*_{B(K(z))}\Elem(K(z))$ which is equal to $\Aut\,K(z)^2$ by van der Kulk Theorem.
Hence $\Gamma$ is a subgroup of
$\Aut_{K[z]}(K[z,x,y])\subset \Aut\,K^3$. By 
Lemma \ref{tame}, $\Gamma$ is indeed a subgroup of
$\TAut\,K^3$. Since $I={\bf m}\cap K[z]$, it follows that $\Gamma$ is  a subgroup of 
$\TAut_{\bf m}\,K^3$.

Therefore by Lemma \ref{HI}, the group $\TAut_{\bf m}\,K^3$
is not linear, even over a ring.

\end{proof}

\end{document}